\documentclass[11pt]{article}
\usepackage{fullpage}

\usepackage{amsmath,amsthm,amssymb}
\usepackage{multirow}
\usepackage[algo2e,ruled]{algorithm2e}
\usepackage[colorlinks,linkcolor=red,citecolor=blue,urlcolor=blue]{hyperref}

\numberwithin{table}{section}
\newtheorem{theorem}{Theorem}[section]
\newtheorem{corollary}[theorem]{Corollary}
\newtheorem{lemma}[theorem]{Lemma}

\newtheorem{assumption}[theorem]{Assumption}
\newtheorem{remark}[theorem]{Remark}

\newcommand{\be}{\begin{equation}}
\newcommand{\ee}{\end{equation}}
\newcommand{\bee}{\begin{equation*}}
\newcommand{\eee}{\end{equation*}}
\newcommand{\bea}{\begin{eqnarray}}
\newcommand{\eea}{\end{eqnarray}}
\newcommand{\beaa}{\begin{eqnarray*}}
\newcommand{\eeaa}{\end{eqnarray*}}

\newcommand{\argmin}{\mathop{\mathrm{arg\,min}{}}}

\newcommand{\xstar}{x^{k^*}_{t^*+1}}
\newcommand{\xistar}{\xi^{k^*}_{t^*}}
\newcommand{\ximstar}{\xi^{k^*}_{t^*-1}}
\newcommand{\xk}{x^k}
\newcommand{\xke}{x^{k+1}}
 
\newcommand{\xki}{x^k_t}
\newcommand{\xkie}{x^{k}_{t+1}}
\newcommand{\xkiestar}{x^{k^*}_{t^*+1}}

\newcommand{\txkm}{\tx^{k-1}} 
\newcommand{\tgk}{\tilde{g}^{k-1}}
\newcommand{\thk}{\tilde{H}^{k-1}}
\newcommand{\cbki}{\cB^k_t}
\newcommand{\cski}{\cS^k_t}
\newcommand{\tgki}{g^k_t}
\newcommand{\thki}{H^k_t}
\newcommand{\xiki}{\xi^{k}_t}
\newcommand{\xikim}{\xi^{k}_{t-1}}
\newcommand{\tx}{\tilde{x}}
\newcommand{\tg}{\tilde{g}}
\newcommand{\thh}{\tilde{H}}

\newcommand{\R}{\mathbb{R}}
\newcommand{\E}{\mathbb{E}} %{\mathbb{E}}

\newcommand{\cB}{\mathcal{B}}

\newcommand{\cS}{{\mathcal{S}}}
\newcommand{\cO}{{\mathcal{O}}}

\newcommand{\ttwo}{^{3/2}}
\newcommand{\tfor}{^{3/4}}

\newcommand{\half}{\frac{1}{2}}

\begin{document}

%\title{Improved Analysis for Adaptive and  Stochastic Variance Reduced Cubic Regularized Newton's Method}
\title{Adaptive Stochastic Variance Reduction for Subsampled Newton Method with Cubic Regularization}

\author{
Junyu Zhang\thanks{Department of Industrial and System Engineering, University of Minnesota (zhan4393@umn.edu).} 
\and
Lin Xiao\thanks{Machine Learning and Optimization Group, Microsoft Research,
Redmond, WA (lin.xiao@microsoft.com).}
\and
Shuzhong Zhang\thanks{Department of Industrial and System Engineering, University of Minnesota (zhangs@umn.edu).}
}

\date{\today} 
\maketitle 

\begin{abstract} 
The cubic regularized Newton method of Nesterov and Polyak has become 
increasingly popular for non-convex optimization because of its capability
of finding an approximate local solution with second-order guarantee.
Several recent works extended this method to the setting of minimizing 
the average of~$N$ smooth functions by replacing the exact gradients 
and Hessians with subsampled approximations.
It has been shown that the total Hessian sample complexity can be reduced
to be sublinear in~$N$ per iteration 
by leveraging stochastic variance reduction techniques.
We present an adaptive variance reduction scheme for subsampled Newton method
with cubic regularization, and show that the expected Hessian sample 
complexity is $\mathcal{O}(N+N^{2/3}\epsilon^{-3/2})$ for finding 
an $(\epsilon,\sqrt{\epsilon})$-approximate local solution 
(in terms of first and second-order guarantees respectively).
Moreover, we show that the same Hessian sample complexity retains
with fixed sample sizes if exact gradients are used. The techniques of our analysis are different from previous works in that we do not
rely on high probability bounds based on matrix concentration inequalities.
Instead, we derive and utilize bounds on the 3rd and 4th order moments of 
the average of random matrices, which are of independent interest on their own.

\paragraph{Keywords:}
cubic-regularized Newton method, subsampling, 
stochastic variance reduction, 
randomized algorithm, 
iteration complexity, sample complexity.
\end{abstract} 
 
\section{Introduction} \label{sec:intro}

We consider the problem of minimizing the average of a large number of loss 
functions:
\be
\label{prob:main-finite}
\min_{x\in\R^d} ~F(x):=\frac{1}{N}\sum_{i=1}^N f_i(x),
\ee
where each $f_i:\R^d\to\R$ is smooth but may be non-convex.
Such problems often arise in machine learning applications where each~$f_i$
is the loss function associated with a training example, and the number of
training examples~$N$ can be very large. 
The necessary conditions for a point $x^*$ to be a local minimum of~$F$ are
\[
\nabla F(x^*)=0, \qquad \nabla^2 F(x^*)\succeq 0.
\]
Our goal is to find an approximate local solution $x$ that satisfies
\begin{equation}
\label{def:2nd-stationary-point}
\|\nabla F(x)\|\leq \epsilon, \qquad
\lambda_\mathrm{min}\left(\nabla^2 F(x) \right) \geq -\sqrt{\epsilon},
\end{equation}
where $\epsilon$ is an desired tolerance and $\lambda_\mathrm{min}(\cdot)$
denotes the smallest eigenvalue of a symmetric matrix.

For minimizing a general smooth and nonconvex function $F$,
Nesterov and Polyak \cite{Cubic:Nesterov} introduced a modified Newton method
with cubic regularization (CR). 
Specifically, each iteration of the CR method consists of the following updates:
\begin{align}
\xi^k &= \displaystyle \arg\min_{\xi} ~\Bigl\{ \xi^T g^k + \half \xi^T H^k \xi + \frac{\sigma}{6}\|\xi\|^3 \Bigr\},
\label{eqn:CR-sub} \\
\xke &= \xk+\xi^k, \nonumber
\end{align} 
where
\begin{equation}\label{eqn:full-g-H}
    g^k = \nabla F(x^k), \qquad H^k = \nabla^2 F(x^k).
\end{equation}
Assuming the Hessian $\nabla^2 F$ to be Lipschitz continuous, it is shown 
in \cite{Cubic:Nesterov} that the CR method finds an approximate solution
satisfying~\eqref{def:2nd-stationary-point} within 
$\mathcal{O}(\epsilon^{-3/2})$ iterations.
This is better than purely gradient-based methods, which need
$\mathcal{O}(\epsilon^{-2})$ iterations to reach a point $x$ satisfying 
$\|\nabla F(x)\|\leq\epsilon$ \cite[Section~1.2.3]{Nesterov04book}.
However, the computational cost per iteration of CR can be much higher
than gradient-based methods.

Much recent efforts have been devoted to improving the efficiency of CR by 
exploiting the finite-sum structure in~\eqref{prob:main-finite}; see, e.g.,
\cite{Cubic:Cartis-1,Cubic:Cartis-2,SCR-1,SCR-2,
AgarwalAllenZhu2017,XuMahoney2017InexactHessian,SVRC-Zhou,SVRC-Lan,
ReddiZaheer2018aistat}.
An natural approach is to replace 
$\nabla F(\xk)$ and $\nabla^2 F(\xk)$ by subsampled approximations:
\begin{align}
g^k = \frac{1}{|\mathcal{S}_k|}\sum_{i\in\mathcal{S}_k}\nabla f_i(\xk),
\label{eqn:grad-sampling}
\end{align}	 

\begin{align} 
H^k = \frac{1}{|\mathcal{B}_k|}\sum_{i\in\mathcal{B}_k}\nabla^2 f_i(\xk),
\label{eqn:Hess-sampling}
\end{align}	
where $\mathcal{S}_k,\mathcal{B}_k\subseteq\{1,\ldots,N\}$ are two sets
(or multisets for sampling with replacement) 
of random indices at the $k$th iteration.
The cost of computing the Hessians $\nabla^2 f_i$ usually dominates
that of the gradients.
Moreover, the cost of solving the CR subproblem~\eqref{eqn:CR-sub} may
grow fast when the batch size $|\mathcal{B}_k|$ increases,
especially when using iterative methods such as gradient descent 
or the Lanczos method 
\cite{Cubic:Cartis-1,Bianconcini2015,Cubic:Carmon,AgarwalAllenZhu2017,chi-jin}.
Therefore, an important measure of efficiency is 
the number of second-order oracle calls for $\nabla^2 f_i$,
i.e., the Hessian sample complexity.

In this paper, we develop an adaptive subsampling CR method that requires
$\mathcal{O}(N+N^{2/3}\epsilon^{-3/2})$ second-order oracle calls 
in expectation.
Assuming that $\epsilon$ is small enough, we often simply refer to it as
$\mathcal{O}(N^{2/3}\epsilon^{-3/2})$.
Notice that using the choices in~\eqref{eqn:full-g-H} would require 
$\mathcal{O}(N\epsilon^{-3/2})$ Hessian samples.
Thus this is a significant improvement especially when $N$ is very large.
In the rest of this section, we discuss several related work, and then
outline our contributions.

\subsection{Related works}

It is shown in \cite{Cubic:Cartis-1} that the order of convergence rate of the CR method
remains the same as long as $g_k$ and $H_k$ in~\eqref{eqn:CR-sub} satisfy 
\begin{align}
\bigl\|g^k-\nabla F(\xk)\bigr\|     &\leq C_1\|\xi_k\|^2, 
\label{eqn:grad-approx-cond}\\
\bigl\|H^k-\nabla^2 F(\xk)\bigr\| &\leq C_2\|\xi_k\|,   
\label{eqn:Hess-approx-cond}
\end{align}
with $\xi_k$ being defined in~\eqref{eqn:CR-sub} and $C_1,C_2$ being some positive constants. 
Here $\|\cdot\|$ for vectors denotes their Euclidean norm,
and for matrices denotes their spectral norm.
In order to exploit the averaging structure in~\eqref{prob:main-finite}, 
a~subsampled cubic regularization (SCR) method was proposed in~\cite{SCR-1}
where $g^k$ and $H^k$ are calculated as in~\eqref{eqn:grad-sampling} and~\eqref{eqn:Hess-sampling}
(here we omit additional accept/reject steps based on trust-region methods).
Matrix concentration inequalities are used in \cite{SCR-1} to derive
appropriate sample sizes $|\mathcal{S}_k|$ and $|\mathcal{B}_k|$
such that the conditions~\eqref{eqn:grad-approx-cond} 
and~\eqref{eqn:Hess-approx-cond} hold with high probability.
In particular, matrix Bernstein inequality (e.g., \cite{Tropp2015,SCR-1})
implies  that with probability at least $1-\delta$, 
\begin{equation}\label{eqn:Hess-concentration}
\bigl\|H^k-\nabla^2 F(x^k)\bigr\|\leq 4 L\sqrt{\frac{\log(2d/\delta)}{|\mathcal{B}_k|}},
\end{equation}
where $L$ is a uniform Lipschitz constant of $\nabla f_i$ for all~$i$.
Therefore, if we upper bound the right-hand side above by $C_2\|\xi_k\|$, 
then~\eqref{eqn:Hess-approx-cond} holds with probability at least $1-\delta$
provided that
\begin{equation}\label{eqn:Hess-samplesize-whp}
|\mathcal{B}_k| \geq \frac{16L^2\log(2d/\delta)}{(C_2\|\xi^k\|)^2}.
\end{equation}
A similar condition on $|\mathcal{S}_k|$ is also derived in~\cite{SCR-1}.
The overall gradient and Hessian sample complexities for SCR, 
with or without replacement, are summarized in
Table~\ref{Table:complexity}
(based on the analysis in \cite{SCR-1,XuMahoney2017InexactHessian,SVRC-Lan}).
When $\epsilon\leq 1/N$, SCR can be much worse than
the deterministic CR method.

\begin{table}[t]
\renewcommand{\arraystretch}{1.5}
	\centering
	\begin{tabular}{|c|c|c|c|c|}  
		\hline
		Algorithm &  Replacement  &Gradient Samples & Hessian  Samples & Convergence Type\\ 
		\hline
        CR \cite{Cubic:Nesterov} & (all samples) & $\mathcal{O}\bigl(\frac{N}{\epsilon^{3/2}}\bigr)$& $\mathcal{O}\bigl(\frac{N}{\epsilon^{3/2}}\bigr)$& Deterministic \\ 
		\hline
        \multirow{2}{*}{SCR \cite{SCR-1}} & with & $\widetilde{\mathcal{O}}\bigl(\frac{1}{\epsilon^{7/2}}\bigr)$& $\widetilde{\mathcal{O}}\bigl(\frac{1}{\epsilon^{5/2}}\bigr)$& w.p. $1-\delta_0$\\
        \cline{2-5}
        & without & $\widetilde{\mathcal{O}}\bigl(\frac{\min\{N,\epsilon^{-2}\}}{\epsilon^{3/2}}\bigr)$ & $\widetilde{\mathcal{O}}\bigl(\frac{\min\{N,\epsilon^{-1}\}}{\epsilon^{3/2}}\bigr)$& w.p. $1-\delta_0$\\
		\hline
        \multirow{2}{*}{SVRC \cite{SVRC-Lan}} & with & Not Provided& $\widetilde{\mathcal{O}}\bigl(N+\frac{N^{3/4}}{\epsilon^{3/2}}\bigr)$ & w.p. $1-\delta_0$\\ 
        \cline{2-5}
& without & Not Provided& $\widetilde{\mathcal{O}}\bigl(N+\frac{N^{8/11}}{\epsilon^{3/2}}\bigr)$ & w.p. $1-\delta_0$\\ 
		\hline
        SVRC \cite{SVRC-Zhou} & with & $\widetilde{\mathcal{O}}\bigl(N+\frac{N^{4/5}}{\epsilon^{3/2}}\bigr)$ & $\widetilde{\mathcal{O}}\bigl(N+\frac{N^{4/5}}{\epsilon^{3/2}}\bigr)$ & In Expectation\\
		\hline 
        \multirow{2}{*}{This Paper}       & without & $\mathcal{O}\bigl(N+\frac{N^{2/3}\min\{N^{1/3},\epsilon^{-1}\}}{\epsilon^{3/2}}\bigr)$ & $\mathcal{O}\bigl(N+\frac{N^{2/3}}{\epsilon^{3/2}}\bigr)$ & In Expectation\\ 
        \cline{2-5}
        & with  & $\cO\bigl(N + N^{2/3}\epsilon^{-5/2}\bigr)$ & $\mathcal{O}\bigl(N+\frac{N^{2/3}}{\epsilon^{3/2}}\bigr)$ & In Expectation\\ 
		\hline
	\end{tabular}
\caption{Comparison of gradient and Hessian sample complexities.
The $\widetilde{\mathcal{O}}(\cdot)$ notation hides poly-logarithmic terms 
such as $\log(d/\epsilon\delta_0)$ for \cite{SCR-1,SVRC-Lan} 
and $\log(d)$ for \cite{SVRC-Zhou}.
%\textcolor{red}{(Please check correctness!)}
}
\label{Table:complexity}
\end{table}

In order to further reduce the Hessian sample complexity,
several recent works \cite{SVRC-Lan,SVRC-Zhou,SCR-2}
combine CR with stochastic variance reduction techniques.
Stochastic variance-reduced gradient (SVRG) method was first proposed to
reduce gradient sample complexity of randomized first-order algorithms 
(see \cite{JohnsonZhang13,ZhangMahdaviJin13,XiaoZhang14ProxSVRG} 
for convex optimization
and \cite{AllenZhuHazan2016ICML,Reddi2016ICML} for nonconvex optimization).
Two different SVRC (stochastic variance-reduced cubic regularization) methods
were proposed in \cite{SVRC-Lan} and \cite{SVRC-Zhou} respectively.
Both of them employed the same variance-reduction technique to reduce the 
Hessian sample complexity. 
In particular, \cite{SVRC-Zhou} incorporated additional 
second-order corrections in gradient variance-reduction, therefore it obtained
better gradient sample complexity, but with slightly worse Hessian sample 
complexity than \cite{SVRC-Lan}.
See Table~\ref{Table:complexity} for a summary of their results.
All these methods require $\mathcal{O}(\epsilon^{-3/2})$ calls to solve the 
cubic regularized sub-problem~\eqref{eqn:CR-sub}. 

Subsampled Newton methods without CR have been studied in, e.g., 
\cite{ErdogduMontanari2015,RoostaMahoney2016,XuYangMahoney2016,YeLuoZhang2017},
but their convergence rates are worse than the ones that are based on CR.

Among the works on CR with stochastic variance reduction, most of them 
(e.g., \cite{SCR-1,SVRC-Lan,SCR-2})
rely on matrix concentration bounds such as~\eqref{eqn:Hess-samplesize-whp}
to set the sample sizes $|\mathcal{B}_k|$ and $|\mathcal{S}_k|$ according
to $\|\xi^k\|$ at each iteration~$k$, 
The problem is that $\xi^k$ is the solution to the CR sub-problem 
in~\eqref{eqn:CR-sub}, where $g^k$ and $H^k$ need to be obtained from
$|\mathcal{S}_k|$ and $|\mathcal{B}_k|$ samples in the first place.
While one can assume bounds like~\eqref{eqn:Hess-samplesize-whp} hold in 
the complexity analysis, 
they do not provide practically implementable variance-reduction schemes.

In contrast, the sample sizes in \cite{SVRC-Zhou} are set as constants 
across all iterations, which only depend on~$N$ and~$\log(d)$.
The disadvantage of this approach is that it can be very conservative.
According to concentration bounds such as~\eqref{eqn:Hess-samplesize-whp},
the sample sizes at the initial stage of the algorithm 
(when $\|\xi^k\|$ is large) can be set much smaller than the ones required
in the later stage (when $\|\xi^k\|$ is very small).
Therefore when using a constant sample size,
much of the samples in the initial stage can be wasteful.

There is another technicality of using matrix concentration bounds:
we need inequalities such as~\eqref{eqn:Hess-concentration} to hold for all 
iterations with high probability, say with probability at least $1-\delta_0$.
Then the probability margin $\delta$ per iteration and $\delta_0$ 
should satisfy $(1-\delta)^T\geq1-\delta_0$, 
where $T = \mathcal{O}(\epsilon^{-3/2})$ is the number of iterations
for CR methods.
Therefore, $\delta = \mathcal{O}(\delta_0\epsilon^{3/2})$
and there is an additional 
$\mathcal{O}\left(\log(d/\epsilon\delta_0)\right)$ factor in the
sampling complexities (see Table~\ref{Table:complexity}).
The results in \cite{SVRC-Zhou} are for convergence in expectation,
nevertheless they still have a $\log(d)$ factor and unusually large constants in their bounds.

\subsection{Contributions and outline}

In this paper, we develop an adaptive sampling scheme for stochastic variance 
reduction in the subsampled Newton method with cubic regularization.
In particular, the gradient sample size $|\mathcal{S}_k|$ and Hessian sample
size $|\mathcal{B}_k|$ are chosen adaptively in each iteration to ensure the
following conditions hold \emph{in expectation} (conditioned on $x^k$):
\begin{align}
\bigl\|g^k-\nabla F(\xk)\bigr\|   &\leq C'_1\|\xi^{k-1}\|^2,
\label{eqn:grad-approx-pre} \\
\bigl\|H^k-\nabla^2 F(\xk)\bigr\| &\leq C'_2\|\xi^{k-1}\|,
\label{eqn:Hess-approx-pre}
\end{align}
where $C'_1$ and $C'_2$ are some positive constants.
The major difference from~\eqref{eqn:grad-approx-cond} 
and~\eqref{eqn:Hess-approx-cond} is that here $\|\xi^{k-1}\|$ is a known 
quantity conditioned on $\xk$, 
before choosing the sample sizes to form the approximations $g^k$ and $H^k$.
Indeed we choose $|\mathcal{S}_k|$ and $|\mathcal{B}_k|$ 
based on $\|\xi^{k-1}\|$.
Such an adaptive scheme is readily implementable in practice\footnote{%
Right before submitting this paper, we discovered a recent note \cite{SVRC-note}
which independently showed that the conditions~\eqref{eqn:grad-approx-cond} 
and~\eqref{eqn:Hess-approx-cond} are sufficient to retain the same convergence 
rate of the exact CR method. However it does not provide improved sample 
complexity over previous work listed in Table~\ref{Table:complexity}.}.
Moreover, it does not waste samples in the early stage of the algorithm
as constant sample sizes do. 

We show that our adaptive subsampled CR method has an expected iteration
complexity $\mathcal{O}(\epsilon^{-3/2})$, which is the number of times
the CR subproblem in~\eqref{eqn:CR-sub} needs to be solved in order to
find a point~$x$ satisfying~\eqref{def:2nd-stationary-point}, which is the same as that of the deterministic CR method.
However, the total Hessian sample complexity of our method is
$\mathcal{O}(N^{2/3}\epsilon^{-3/2})$, which is much better than
$\mathcal{O}(N\epsilon^{-3/2})$ of the full CR method, and is indeed better than
all previous works listed in Table~\ref{Table:complexity}.

In addition to the improved Hessian sample complexity, the techniques in our analysis
are quite different from those adopted in previous works.
In particular, we avoid using any high probability bounds based on
matrix or vector concentration inequalities.
Instead, our analysis is based on novel bounds on the 3rd and 4th order 
moments of the average of independent random matrices, 
which are of independent interest on their own.
The type of convergence studied in \cite{SVRC-Zhou} is also in expectation.
However, their analysis still relies on some matrix concentration inequalities,
thus their results contain the $\log(d)$ factor and excessively large constants.

The rest of this paper is organized as follows.
In Section~\ref{sec:adaptive-SVRC}, we present the adaptive SVRC method
and its convergence analysis.
We show that it retains the $\cO(\epsilon^{-3/2})$ iteration complexity of
the exact cubic regularization method, but with only $\cO(N^{2/3})$ 
Hessian samples per iteration.
In addition, we show that sampling with or without replacement have the 
same order of sample complexity.
In Section~\ref{sec:non-adaptive}, we study a non-adaptive SVRC method
with fixed sample size at each iteration. We show that if exact gradients 
are available, then it attains the same total Hessian sample complexity 
$\cO(N^{2/3}\epsilon^{-3/2})$.
If both gradient and Hessian need to be subsampled, we examine the SVRC
method of~\cite{SVRC-Zhou} using the higher moments bounds developed in 
this paper and obtain refined analysis.

\section{Adaptive variance reduction for cubic regularization}
\label{sec:adaptive-SVRC}

In this section we first present the adaptive SVRC method, then analyze its
convergence rate and Hessian sample complexity.

The adaptive SVRC method is a multi-stage iterative method as described in
Algorithm~\ref{algo:SVRC}. 
At the beginning of each stage~$k\geq 1$, we compute the full gradient
$\tgk$ and the full Hessian $\thk$ at $\txkm$, which is the result of the 
previous stage (for $k=1$, $\tilde{x}^0$ is given as input).
Each stage~$k$ has an inner loop of length~$m$ and the variables used in the 
inner loop are indexed with a subscript~$t=0,\ldots,m$.
At the beginning of each inner loop, we set $x^k_0=\txkm$ and $\xi^k_{-1}=0$.
During each inner iteration, we randomly sample two set of indices 
$\cski$ and $\cbki$ satisfying
\begin{equation}
\label{eqn:sample-sizes}
|\cski|\geq \frac{\|\xki-\txkm\|^2}{\epsilon_g}, \qquad
|\cbki|\geq \frac{\|\xki-\txkm\|^2}{\epsilon_H},
\end{equation}
where $\epsilon_g$ and $\epsilon_H$ are determined by $\xikim$ and the desired
tolerance $\epsilon$:
\[
\epsilon_g = \max\bigl\{\|\xikim\|^4, \epsilon^2\bigr\}, \qquad
\epsilon_H = \max\bigl\{\|\xikim\|^2, \epsilon\bigr\}.
\] 
%\textcolor{red}{($\cski$ and $\cbki$ need to be independent?)}
Then we compute the subsampled gradient and Hessian as
\begin{align}
    \tgki &= \frac{1}{|\cski|}\sum_{i\in\cski}\left(\nabla f_i(\xki)-\nabla f_i(\txkm)\right) + \tgk, \label{eqn:svr-g} \\
    \thki &= \frac{1}{|\cbki|}\sum_{i\in\cbki}\left(\nabla^2 f_i(\xki)-\nabla^2 f_i(\txkm)\right) + \thk.
\label{eqn:svr-H}
\end{align}
This construction follows the stochastic variance reduction scheme proposed
in~\cite{JohnsonZhang13}, which has been adopted 
by recent works on subsampled Newton method with cubic regularization 
\cite{chi-jin,SVRC-Lan,SVRC-Zhou}.

We make several remarks regarding the subsampling scheme in 
Algorithm~\ref{algo:SVRC}.  
First, compared with the algorithms in \cite{SCR-1, SVRC-Lan}, 
which use large enough sample sizes to 
ensure~\eqref{eqn:grad-approx-cond} and~\eqref{eqn:Hess-approx-cond} 
with high probability, 
Algorithm~\ref{algo:SVRC} aims to ensure~\eqref{eqn:grad-approx-pre} 
and~\eqref{eqn:Hess-approx-pre} in expectation. 
Moreover, instead of depending on $\xiki$ which is not available until after
the current iteration, our sample sizes are determined by $\xikim$, 
which is computed in the previous iteration.
Second, compared with the constant sample size used in \cite{SVRC-Zhou}, 
our adaptive sampling scheme may use much less samples in the early stages
when $\|\xikim\|$ is relatively large.
Our overall Hessian sample complexity is better than either of these two
previous approaches.

An alternative approach to avoid the dependence of sample sizes on $\xiki$
is to use the full gradient, i.e., let $g_t^k=\nabla F(x_t^k)$,
and use~\eqref{eqn:svr-H} or \eqref{eqn:Hess-sampling}
to compute the Hessian approximation \cite{SCR-2}.
Then condition~\eqref{eqn:grad-approx-cond} holds automatically, and one can
replace~\eqref{eqn:Hess-approx-cond} with
\begin{equation}\label{eqn:Hess-approx-grad}
\bigl\|\thki - \nabla^2 F(\xki)\bigr\|\leq C_2\bigl\|\nabla F(\xki)\bigr\|,
\end{equation}
because it can be be shown that $\|\xiki\|\leq c\|\nabla F(\xki)\|$ 
for some large enough constant~$c$.
Since the full gradient $\nabla F(\xki)$ can be computed before sampling 
$\cbki$, this condition can be used to determine sample size $|\cbki|$. 
However, 
%in a not rigorous sense one have 
it can be shown that one have roughly 
$\|\nabla F(\xki)\| \leq \cO(\|\xikim\|^2)$ plus some random noise, thus the mini-batch size required for~\eqref{eqn:Hess-approx-grad} to hold
can be much larger than that for \eqref{eqn:Hess-approx-pre}. 
 
\begin{algorithm2e}[t]
\linespread{1.2}\selectfont
	%\caption{Adaptive Stochastic Variance Reduced Cubic Regularization Method.}
	\caption{Adaptive SVRC Method.}
	\label{algo:SVRC}%\LinesNumberedHidden
    {\it Input:} initial point $\tilde{x}^0$, inner loop length $m$, 
                 CR parameter $\sigma>0$, and tolerance $\epsilon>0$. \\
 	\For{k = 1,......,K }{
 		Compute $\tgk = \nabla F (\txkm)$ and $\thk = \nabla^2 F(\txkm)$. \\
 		Assign $x^k_0 = \txkm$ and $\xi^k_{-1}=0$. \\
		\For{$t = 0,...,m-1$}{
			Set $\epsilon_g = \max\bigl\{\|\xikim\|^4, \epsilon^2\bigr\}$, $\epsilon_H = \max\bigl\{\|\xikim\|^2, \epsilon\bigr\}$.\\
            Sample index sets $\cski$ and $\cbki$ 
            %satisfying~\eqref{eqn:sample-sizes}. \\
            with $|\cski|\geq \frac{\|\xki-\txkm\|^2}{\epsilon_g}$ and $|\cbki|\geq \frac{\|\xki-\txkm\|^2}{\epsilon_H}$.\\ 
            Construct $\tgki$ and $\thki$ according to~\eqref{eqn:svr-g} and~\eqref{eqn:svr-H} respectively.\\
            $\xiki = \displaystyle\arg\min_{\xi} \Bigl\{\xi^T\tgki + \frac{1}{2}\xi^T\thki\xi + \frac{\sigma}{6}\|\xi\|^3\Bigr\}. $\\		
		$\xkie = \xki + \xiki.$}
	$\tx^k = x^k_m.$
	}
{\it Output:}
Option 1:  
Find $\displaystyle k^*,t^*=\argmin_{1\leq k\leq K, ~0\leq t\leq m-1}\bigl(\|\xiki\|^3+\|\xikim\|^3\bigr)$ and output $x^{k^*}_{t^*+1}$. \\
$\qquad\quad~\,$ 
Option 2: Randomly choose $k^*$ and $t^*$ and output $x^{k^*}_{t^*+1}$.
\end{algorithm2e}

\subsection{Convergence analysis} \label{subsec:cvg-finite}

We make the following assumption regarding the objective function 
in~\eqref{prob:main-finite}:
\begin{assumption}
	\label{assumption:Lips-finite}
    The gradient and Hessian of each component function~$f_i$ are
	Lipschitz-continuous, i.e., 
    there exist positive constants~$L$ and~$\rho$ such that 
    for $i=1,\ldots, N$ and for all $x, y\in\R^d$,
\begin{align}
    \bigl\|\nabla f_i(x)-\nabla f_i(y) \bigr\|  \leq L \|x-y\|, 
      \label{assumption: gradient-Lip} \\
    \bigl\|\nabla^2f_i(x)-\nabla^2f_i(y)\bigr\|_F  \leq \rho\|x-y\|.
	  \label{assumption: Hessian-Lip}
\end{align}
    Consequently, $\nabla F$ and 
    $\nabla^2 F$ are $L$ and $\rho$-Lipschitz continuous respectively.
\end{assumption}

This assumption is very similar to those adopted in subsampled Newton methods
with cubic regularization \cite{Cubic:Nesterov,SCR-1,SCR-2,SVRC-Lan,SVRC-Zhou}.
The only difference is that here we use the Frobenius norm, instead of the
spectral norm, for the Hessian smoothness assumption.
The advantage of using the Frobenius norm is that it works well with matrix
inner product, which allows us to derive simple bounds on the 3rd and 4th order
moments for the average of random matrices.
On the other hand, the Lipschitz constant $\rho$ is always larger than the
one corresponding to the spectral norm, up to a factor of $\sqrt{d}$ in 
the worst case. However, when the Hessians $\nabla^2 f_i$ are of low-rank
or ill-conditioned, which is often the case in practice, 
the Lipschitz constants for different norms are very close.

Before presenting the main results, we first provide a few supporting lemmas. 
The first one gives a simple bound on the 4th order moment of the average of
i.i.d.\ random matrices with zero mean.
Its proof is given in Appendix~\ref{appendix:iid4thmoment}.
\begin{lemma}
	\label{lemma:var-with-pw4}
    Let $Z_1,...,Z_n$ be i.i.d.\ random matrices in $\R^{d\times d}$ with $\mathbb{E}\left[Z_1\right]=0$ and $\mathbb{E}\left[\|Z_1\|_F^4\right]<\infty.$ Then
	\begin{equation}
	\label{lm:var-with-pw4-1}
	\mathbb{E}\left[\biggl\|\frac{1}{n}\sum\limits_{i=1}^{n}Z_i\biggr\|_F^4\right] \leq \frac{3}{n^2}\mathbb{E}\left[\|Z_1\|_F^4\right].
	\end{equation} 
\end{lemma}
Based on the above lemma, we can bound the 2nd and 4th order variances of the 
gradient and Hessian approximations.
The following lemma is proved in Appendix~\ref{appendix:Hess-grad-var}.
\begin{lemma}
	\label{lemma:var-with-Hess-g}
    Let the variance reduced gradient $\tgki$ and  Hessian $\thki$ be constructed according to~\eqref{eqn:svr-g} and~\eqref{eqn:svr-H}. Then they satisfy the following equalities and inequalities
	\begin{align*}
        \mathbb{E}\bigl[\thki\,\big|\,\xki\bigr]  &= \nabla^2 F(\xki),\\
        \mathbb{E}\left[\bigl\|\thki-\nabla^2 F(\xki)\bigr\|_F^2\,\big|\,\xki\right] &\leq \frac{\rho^2}{|\cbki|} \bigl\|\xki-\txkm\bigr\|^2,\\
        \mathbb{E}\left[\bigl\|\thki-\nabla^2 F(\xki)\bigr\|_F^4\,\big|\,\xki\right] &\leq \frac{33\rho^4}{|\cbki|^2} \bigl\|\xki-\txkm\bigr\|^4,
	\end{align*}
	and 
	\begin{align*}
        \mathbb{E}\left[\tgki\,\big|\,\xki\right]  &= \nabla F(\xki),\\
        \mathbb{E}\left[\bigl\|\tgki-\nabla F(\xki)\bigr\|_F^2 \,\big|\,\xki\right] &\leq \frac{L^2}{|\cski|} \bigl\|\xki-\txkm\bigr\|^2.
    \end{align*}
\end{lemma}
As a result, we have the following corollary, whose proof is given in
Appendix~\ref{appendix:var-Hess-grad-xi}.
\begin{corollary}
	\label{corollary:var-with-Hess-g-xi}
	Let $\thki$, $\tgki$, $\xiki$ and the mini-batch index sets $\cbki$ and $\cski$ be generated according to Algorithm~\ref{algo:SVRC}. Then we have 
	\begin{align*}
        \mathbb{E}\left[\bigl\|\thki-\nabla^2 F(\xki)\bigr\|_F \,\big|\, \xki\right] &\leq \rho\left(\|\xikim\| + \epsilon^{1/2}\right),\\
        \mathbb{E}\left[\bigl\|\thki-\nabla^2 F(\xki)\bigr\|_F^2 \,\big|\, \xki\right] &\leq \rho^2\left(\|\xikim\|^2 + \epsilon\right),\\
        \mathbb{E}\left[\bigl\|\thki-\nabla^2 F(\xki)\bigr\|_F^3 \,\big|\, \xki\right] &\leq 33^{3/4}\rho^3\left(\|\xikim\|^3 + \epsilon^{3/2}\right),
	\end{align*}
	and 
	\begin{align*}
        \mathbb{E}\left[\bigl\|\tgki-\nabla F(\xki)\bigr\|_F \,\big|\, \xki\right] &\leq L\left(\|\xikim\|^2 + \epsilon\right), \\
        \mathbb{E}\left[\bigl\|\tgki-\nabla F(\xki)\bigr\|_F^{3/2} \,\big|\, \xki\right] &\leq L^{3/2}\left(\|\xikim\|^3 + \epsilon^{3/2}\right).
	\end{align*}
\end{corollary}

Now we are ready to present the descent property of Algorithm~\ref{algo:SVRC}.
\begin{lemma}
\label{lemma:descent-finite}
Suppose the sequence $\{\xki\}^{k = 1,...,K}_{i = 1,...,m}$ is generated by Algorithm \ref{algo:SVRC}. Then the following descent property holds  
\be
\label{lm:descent-finite}
\E\left[F(\xkie)\Big|\xki\right] 
\leq F(\xki)-\left(\frac{\sigma}{4} - \frac{\rho}{2} - \frac{L}{3}\right)
\E\left[\bigl\|\xiki\bigr\|^3\Big|\xki\right] 
+ \left(\frac{5\rho}{2}+\frac{2L}{3}\right)
  \left(\bigl\|\xikim\bigr\|^3+\epsilon^{3/2}\right).
\ee
\end{lemma}

As a remark, as long as 
$\frac{\sigma}{4}-\frac{\rho}{2}-\frac{L}{3} > \frac{5\rho}{2} + \frac{2L}{3}$,
the variance term associated with $\|\xikim\|^3$ in current step can be dominated by the descent associated
with $\|\xikim\|^3$ in the previous step. 
At the first step of each stage, the variance term is 0, 
hence we define $\xi^k_{-1}=0$ for a unified expression.

\begin{proof}
Consider the cubic regularization subproblem in Algorithm~\ref{algo:SVRC}:
\[
\xiki = \arg\min_{\xi} ~\Bigl\{\xi^T\tgki + \frac{1}{2}\xi^T\thki\xi + \frac{\sigma}{6}\|\xi\|^3\Bigr\}. 
\]
Its optimality conditions imply (see \cite{Cubic:Nesterov})
\begin{align}
\tgki + \thki\xiki+ \frac{\sigma}{2}\bigl\|\xiki\bigr\|\xiki &= 0, 
\label{eqn:subopt}\\
\thki + \frac{\sigma}{2}\bigl\|\xiki\bigr\|I &\succeq 0. 
\label{eqn:subopt-psd}
\end{align}
Then by the Lipschitz continuous condition of the objective function, we have 
\begin{eqnarray}  
F(\xkie) & \leq & F(\xki) + \nabla F(\xki)^T\xiki + \half(\xiki)^T\nabla^2 F(\xki)\xiki + \frac{\rho}{6}\bigl\|\xiki\bigr\|^3 \nonumber \\
	& = & F(\xki) + \nabla F(\xki)^T\xiki + \half(\xiki)^T\nabla^2 F(\xki)\xiki + \frac{\rho}{6}\bigl\|\xiki\bigr\|^3   - (\xiki)^T\left(\tgki + \thki\xiki + \frac{\sigma}{2}\bigl\|\xiki\bigr\|\xiki\right) \nonumber\\
	& = & F(\xki) + \left(\nabla F(\xki)-\tgki\right)^T\xiki + \half(\xiki)^T\left(\nabla^2 F(\xki)-\thki\right)\xiki - \left(\frac{\sigma}{4} - \frac{\rho}{6}\right)\bigl\|\xiki\bigr\|^3 \nonumber\\
	&&   - \half(\xiki)^T\left(\thki + \frac{\sigma}{2}\bigl\|\xiki\bigr\|I\right)\xiki \nonumber\\
	& \leq & F(\xki) + \bigl\|\xiki\bigr\|\,\bigl\|\tgki-\nabla F(\xki)\bigr\| + \half\bigl\|\nabla^2 F(\xki)-\thki\bigr\|_F\bigl\|\xiki\bigr\|^2  - \left(\frac{\sigma}{4} - \frac{\rho}{6}\right)\bigl\|\xiki\bigr\|^3,
\label{lm:descent-finite-2}
\end{eqnarray}
where the second line is due to~\eqref{eqn:subopt} and the fifth line is due 
to~\eqref{eqn:subopt-psd}. 
By the following variant of Young's inequality 
\[
ab \leq \frac{(a\theta )^p}{p} + \frac{(b/\theta)^q}{q}, \qquad 
\forall\; a,b,p,q,\theta>0 
\quad \mathrm{and}\quad \frac{1}{p} + \frac{1}{q} = 1,
\] 
we have
\beaa
\bigl\|\xiki\bigr\|\,\bigl\|\tgki-\nabla F(\xki)\bigr\| 
& \leq & \frac{\bigl(L^{1/3}\bigl\|\xiki\bigr\|\bigr)^3}{3} + \frac{\bigl(L^{-1/3}\|\nabla F(\xki) - \tgki\|\bigr)^{3/2}}{3/2} \\
& = & \frac{L}{3}\bigl\|\xiki\bigr\|^3 + \frac{2}{3\sqrt{L}}\bigl\|\nabla F(\xki) - \tgki\bigr\|_F^{3/2},
\eeaa
and 
\beaa
\bigl\|\xiki\bigr\|^2\bigl\|\thki-\nabla^2 F(\xki)\bigr\| 
& \leq & \frac{\left(\rho^{2/3}\bigl\|\xiki\bigr\|\right)^{3/2}}{3/2} + \frac{\left(\rho^{-2/3}\bigl\|\nabla^2 F(\xki) - \thki\bigr\|\right)^{3}}{3} \\
& = &\frac{2\rho}{3}\bigl\|\xiki\bigr\|^3 + \frac{1}{3\rho^2}\bigl\|\nabla^2 F(\xki) - \thki\bigr\|^{3}.
\eeaa
Consequently, 
\bea
\E\left[F(\xkie)\big|\xki\right] 
& \leq & F(\xki) - \left(\frac{\sigma}{4} - \frac{\rho}{2} - \frac{L}{3}\right)\E\left[\bigl\|\xiki\bigr\|^3\big|\xki\right] + \frac{1}{6\rho^2}\E\left[\bigl\|\nabla^2 F(\xki)-\thki\bigr\|_F^3\big|\xki\right] \nonumber \\
&& +\, \frac{2}{3\sqrt{L}}\E\left[\bigl\|\nabla F(\xki) - \tgki\bigr\|^{3/2}\big|\xki\right].
\label{lm:descent-finite-1}
\eea
Since $g^k_0 = \nabla F(x^k_0)$ and $H^k_0= \nabla^2 F(x^k_0)$
at the beginning of each stage, we have
\bea
\label{lm:descent-finite-1.0}
\E\left[F(x^k_1)\big|x_0^k\right] 
& \leq & F(x^k_0) - \left(\frac{\sigma}{4} - \frac{\rho}{6}\right)\E\left[\bigl\|\xi_0^k\bigr\|^3\big|x_0^k\right].
\eea
By substituting the variance bounds in Corollary \ref{corollary:var-with-Hess-g-xi} into \eqref{lm:descent-finite-1}, we get
\beaa
\E\left[F(\xkie)\big|\xki\right] 
& \leq & F(\xki) - \left(\frac{\sigma}{4} - \frac{\rho}{2} - \frac{L}{3}\right)\E\left[\bigl\|\xiki\bigr\|^3\big|\xki\right] + \biggl(\frac{33\tfor\rho}{6}+\frac{2L}{3}\biggr)\left(\bigl\|\xikim\bigr\|^3 + \epsilon^{3/2}\right) \nonumber\\
& \leq & F(\xki) - \left(\frac{\sigma}{4} - \frac{\rho}{2} - \frac{L}{3}\right)\E\left[\bigl\|\xiki\bigr\|^3\big|\xki\right] + \biggl(\frac{5\rho}{2}+\frac{2L}{3}\biggr)\left(\bigl\|\xikim\bigr\|^3 + \epsilon^{3/2}\right). \nonumber
\eeaa
This completes the proof.
\end{proof}

Next, we give a lemma that connects 
$\E[\|\nabla f(\xkie)\|]$ and $\E[\lambda_{\min}(\nabla^2f(\xkie))]$ 
with $\xiki$ and $\xikim$, whose proof is given in 
Appendix~\ref{appendix:relation-finite}.
\begin{lemma}
\label{lemma:relation-finite}
Let $\xkie$, $\xikim$ and $\xiki$ be generated by Algorithm~\ref{algo:SVRC}. 
Then the following relations hold
\begin{align}
\E\left[\bigl\|\nabla F(\xkie)\bigr\|\right]
&\leq \left(\rho+\frac{\sigma}{2}\right)\E\left[\bigl\|\xiki\bigr\|^2\right] + \left(\frac{\rho}{2}+L\right)\left(\E\left[\bigl\|\xikim\bigr\|^2\right] + \epsilon\right), \nonumber\\
\mathbb{E}\left[\lambda_{\min}\bigl(\nabla^2f(\xkie)\bigr)\right]  
&\geq -\left(\rho+\frac{\sigma}{2}\right)\mathbb{E}\left[\bigl\|\xiki\bigr\|\right] -\rho\left(\mathbb{E}\left[\bigl\|\xikim\bigr\|\right]+\epsilon^{1/2}\right). \nonumber
\end{align}
\end{lemma}

Finally we are ready to present the main result on iteration complexity.
\begin{theorem}[Iteration Complexity]
\label{theorem:cvg-finite}
Choose the parameter $\sigma> 13\rho + 4L$ and let $k^*,i^*$ be given by 
either output option in Algorithm~\ref{algo:SVRC} after running for $K$ stages,
then
\begin{equation}\label{eqn:two-xi-bound}
\E\left[\bigl\|\xistar\bigr\|^3+\bigl\|\ximstar\bigr\|^3\right]  
~\leq~ \frac{2\left(F(\tx^0) - F^*\right)}{Km(\sigma/4 - 3\rho - L)}
+\frac{5\rho+4L/3}{\sigma/4 - 3\rho - L}\epsilon^{3/2},
\end{equation}
where $F^\star=\min_x F(x)$.
As a result, 
\begin{align}
\E\left[\bigl\|\nabla F(\xkiestar)\bigr\|\right] 
&\leq C^g_1(Km)^{-2/3}+C^g_2\epsilon,
\label{eqn:1st-order-bound}\\
\mathbb{E}\left[\lambda_{\min}\bigl(\nabla^2 F(\xkiestar)\bigr)\right] 
&\geq -C^H_1(Km)^{-1/3} - C^H_2\epsilon^{1/2},
\label{eqn:2nd-order-bound}
\end{align}
where the constants satisfy
\begin{align*}
    C^g_1 &= \cO\left((\rho+L)^{1/3}\left(F(\tx^0)-F^\star\right)^{2/3}\right), 
    \qquad C^g_2 = \cO(\rho+L), \\
    C^H_1 &= \cO\left((\rho+L)^{2/3}\left(F(\tx^0)-F^\star\right)^{1/3}\right),
    \qquad C^H_2 = \cO(\rho+L).
\end{align*}
%and $C^g_2 = \cO(\rho+L)$ and $C^H_2 = \cO(\rho+L)$.
\end{theorem}
\begin{remark}
If we choose $K$ and $m$ such that $Km = \cO\bigl(\epsilon^{-3/2}\bigr)$, 
then within $\cO\bigl(\epsilon^{-3/2}\bigr)$ iterations, 
Algorithm \ref{algo:SVRC} would reach a point $\xstar$ such that 
\[
\E\left[\bigl\|\nabla F(\xstar)\bigr\|\right]\leq \cO(\epsilon),\qquad
\E\left[\lambda_{\min}\bigl(\nabla^2 F(\xstar)\bigr)\right]\geq -\cO\bigl(\sqrt{\epsilon}\bigr).
\]
In other words, the approximate optimality conditions 
in~\eqref{def:2nd-stationary-point} are satisfied in expectation.
\end{remark}
\begin{proof}
First, taking expectation over the whole history of random samples 
for~\eqref{lm:descent-finite} and~\eqref{lm:descent-finite-1.0} 
and summing them up for $t=0,...,m-1$ yield
\begin{eqnarray*}
&&\E\left[F(x_0^k)\right] - \E\left[F(x_m^k)\right] + m\left(\frac{5\rho}{2} + \frac{2L}{3}\right)\epsilon^{3/2} \\
&\geq& \left(\frac{\sigma}{4} - \frac{\rho}{6}\right)\E\left[\bigl\|\xi_0^k\bigr\|^3\right]+\left(\frac{\sigma}{4} - 3\rho - L\right)\sum_{t=1}^{m-2}\E\left[\bigl\|\xiki\bigr\|^3\right] + \left(\frac{\sigma}{4} - \frac{\rho}{2}-\frac{L}{3}\right)\E\left[\bigl\|\xi_{m-1}^k\bigr\|^3\right]\\
& \geq & \left(\frac{\sigma}{4} - 3\rho - L\right)\sum_{t=0}^{m-1}\E\left[\bigl\|\xiki\bigr\|^3\right].
\end{eqnarray*}
Under the assumption $\sigma>13\rho+4L$, we have $\sigma/4-3\rho-L>0$.
Further summing over the stages $k=1,\ldots,K$, we obtain
\begin{equation}\label{eqn:sum-xikt}
\left(\frac{\sigma}{4} - 3\rho - L\right)\sum_{k=1}^K\sum_{t=0}^{m-1}\E\left[\bigl\|\xiki\bigr\|^3\right]
~\leq~ F(\tx^0) - F^\star + Km\left(\frac{5\rho}{2} + \frac{2L}{3}\right)\epsilon^{3/2}.
\end{equation}
For option~1 in the output rule, due to the concavity of $\min$ function, Jensen's inequality gives
\beaa
\E\left[\bigl\|\xistar\bigr\|^3 + \bigl\|\ximstar\bigr\|^3\right] 
& = & \E\Bigl[\min_{i,k}\left(\bigl\|\xiki\bigr\|^3+\bigl\|\ximstar\bigr\|^3\right)\Bigr]\\
&\leq& \min_{i,k}\E\left[\bigl\|\xiki\bigr\|^3+\bigl\|\ximstar\bigr\|^3\right]\\
&\leq& \frac{2\left(F(\tx^0) - F^\star + Km\bigl(5\rho/2 + 2L/3\bigr)\epsilon^{3/2}\right)}{Km\bigl(\sigma/4 - 3\rho - L\bigr)},
\eeaa 
which is precisely~\eqref{eqn:two-xi-bound}.
For option 2, since $k^*$ and $t^*$ are randomly chosen, 
\beaa
\E\left[\bigl\|\xistar\bigr\|^3 + \bigl\|\ximstar\bigr\|^3\right] 
& = &  \frac{2}{Km}\sum_{k=1}^K\sum_{t=0}^{m-1}\E\left[\bigl\|\xiki\bigr\|^3\right]-\frac{1}{Km}\sum_{k=1}^K\E\left[\bigl\|\xi^k_{m-1}\bigr\|^3\right]\\
& \leq & \frac{2\left(F(\tx^0) - F^\star + Km\bigl(5\rho/2 + 2L/3\bigr)\epsilon^{3/2}\right)}{Km(\sigma/4 - 3\rho - L)}.
\eeaa
Therefore, for both options, inequality~\eqref{eqn:two-xi-bound} holds.

Next, we derive guarantees for approximating the first and second-order 
stationary conditions.
According to Lemma \ref{lemma:relation-finite}, 
\begin{equation}\label{eqn:grad-xi-epsilon}
\E\left[\bigl\|\nabla F(\xkiestar)\bigr\|\right]
~\leq~ \left(\rho+\frac{\sigma}{2}\right)\E\left[\bigl\|\xistar\bigr\|^2\right]
 + \left(\frac{\rho}{2}+L\right)\E\left[\bigl\|\ximstar\bigr\|^2\right] 
+ \left(\frac{\rho}{2} + L\right)\epsilon.
\end{equation}
Consequently, we have
\beaa
& & \E\left[ \left(\rho+\frac{\sigma}{2}\right)\bigl\|\xistar\bigr\|^2 + \left(\frac{\rho}{2}+L\right)\bigl\|\ximstar\bigr\|^2\right] \\
& \leq & \left(\left(\rho+\frac{\sigma}{2}\right)^3+\left(\frac{\rho}{2}+L\right)^3\right)^{1/3} \E\left[\left(\left(\|\ximstar\|^2\right)^{3/2}+\left(\|\xistar\|^2\right)^{3/2}\right)^{2/3}\right] \\
& \leq & \left(\left(\rho+\frac{\sigma}{2}\right)^3+\left(\frac{\rho}{2}+L\right)^3\right)^{1/3} \left(\E\left[\bigl\|\ximstar\bigr\|^3+\bigl\|\xistar\bigr\|^3\right]\right)^{2/3} \\
&\leq & \left(\left(\rho+\frac{\sigma}{2}\right)^3+\left(\frac{\rho}{2}+L\right)^3\right)^{1/3} \left(\frac{2(F(\tx^0) - F^\star)}{Km(\sigma/4 - 3\rho - L)}+\frac{5\rho+4L/3}{\sigma/4 - 3\rho - L}\epsilon^{3/2}\right)^{2/3}\\
& \leq & \left(\left(\rho+\frac{\sigma}{2}\right)+\left(\frac{\rho}{2}+L\right)\right) \left(\left(\frac{2(F(\tx^0) - F^\star)}{\sigma/4 - 3\rho -  L}\right)^{2/3}\frac{1}{(Km)^{2/3}} + \left(\frac{5\rho+4L/3}{\sigma/4 - 3\rho - L}\right)^{2/3}\epsilon\right),
\eeaa
where the first inequality is due to H\"older's inequality, 
the second inequality is due to Jensen's inequality, 
the third inequality is due to~\eqref{eqn:two-xi-bound}
and the last inequality is due to the fact that 
$(a+b)^{\theta}\leq a^{\theta}+b^{\theta}$ for all $a,b>0$ and 
$0\leq\theta\leq 1$.
Finally, substituting the above inequality into~\eqref{eqn:grad-xi-epsilon}
yields the desired result in~\eqref{eqn:1st-order-bound}.
	
In order to bound the minimum eigenvalue of the Hessian, we have 
from Lemma \ref{lemma:relation-finite}, 
\begin{equation}\label{eqn:Hess-xi-epsilon}
\mathbb{E}\left[\lambda_{\min}\bigl(\nabla^2 F(\xkiestar)\bigr)\right]  
~\geq~ -\left(\rho+\frac{\sigma}{2}\right)\mathbb{E}\left[\bigl\|\xistar\bigr\|\right] -\rho\mathbb{E}\left[\bigl\|\ximstar\bigr\|\right]- \rho\epsilon^{1/2}.
\end{equation}
Similar to the arguments used for proving the first-order bound, 
\beaa
& & \mathbb{E}\left[\left(\rho+\frac{\sigma}{2}\right)\bigl\|\xistar\bigr\|+ \rho\bigl\|\ximstar\bigr\|\right] \\
& \leq & \left(\left(\rho+\frac{\sigma}{2}\right)^{3/2} + \rho^{3/2}\right)^{2/3}\E\left[\left(\bigl\|\ximstar\bigr\|^3+\bigl\|\xistar\bigr\|^3\right)^{1/3}\right]\\
& \leq & \left(\left(\rho+\frac{\sigma}{2}\right)^{3/2} + \rho^{3/2}\right)^{2/3}\left(\E\left[\bigl\|\ximstar\bigr\|^3+\bigl\|\xistar\bigr\|^3\right]\right)^{1/3}\\
& \leq &  \left(\left(\rho+\frac{\sigma}{2}\right)^{3/2} + \rho^{3/2}\right)^{2/3}\left(\frac{2(F(\tx^0) - F^\star)}{Km(\sigma/4 - 3\rho - L)}+\frac{5\rho+4L/3}{\sigma/4 - 3\rho - L}\epsilon^{3/2}\right)^{1/3}\\
& \leq &  \left(\left(\rho+\frac{\sigma}{2}\right) + \rho\right) \left(\left(\frac{2(F(\tx^0) - F^\star)}{\sigma/4 - 3\rho -  L}\right)^{1/3}\frac{1}{(Km)^{1/3}} + \left(\frac{5\rho+4L/3}{\sigma/4 - 3\rho - L}\right)^{1/3}\epsilon^{1/2}\right).
\eeaa
Combining the inequality above with~\eqref{eqn:Hess-xi-epsilon} gives the 
desired bound in~\eqref{eqn:2nd-order-bound}.
%Hence, 
%\[
%\mathbb{E}\left[\lambda_{\min}\bigl(\nabla^2 F(\xkiestar)\bigr)\right]  
%~\geq~ -C_H(Km)^{-1/3} - C'_H\sqrt{\epsilon},
%\]
%where $C_H = \cO\bigl((\rho+L)^{2/3}\bigl(F(\tx^0)-F^\star\bigr)^{1/3}\bigr)$
%and $C'_H = \cO(\rho+L)$.
\end{proof}

\subsection{Bounding the Hessian sample complexity}
Due to the adaptive mini-batch size rule, the total Hessian sample complexity is not given explicitly. In this subsection, we provide a bound on the complexity of Hessian sampling.  
\begin{theorem} 
\label{theorem:Hess-complexity}
Let the total number of Hessian samples in Algorithm~\ref{algo:SVRC} be $B_H$.  
If we set the length of each stage and number of stages to be  
\begin{equation}
\label{thm:Hess-complexity-m}
m = \mathcal{O}\bigl(N^{1/3}\bigr), \qquad  K = \epsilon^{-3/2}/m,
\end{equation}
then the expectation of $B_H$ taken to reach a second-order $\epsilon$-solution will be 
\begin{equation}
\label{thm:Hess-complexity-0}
\mathbb{E}\left[B_H\right] \leq \mathcal{O}\bigl(N^{2/3}\epsilon^{-3/2}\bigr).
\end{equation}
\end{theorem} 
\begin{proof}
According to~\eqref{eqn:sample-sizes}, it suffices to use the following
sample size for approximating the Hessian,
$$
|\cbki| = \frac{\bigl\|\xki-\txkm\bigr\|^2}{\max\bigl\{\|\xikim\|^2,\epsilon\bigr\}}  \leq \epsilon^{-1}\bigl\|\xki-\txkm\bigr\|^2.
$$
Noticing that $\txkm=x^k_0$ and 
$x^k_t-\txkm=\sum_{j=1}^t(x^k_j-x^k_{j-1})=\sum_{j=0}^{t-1}\xi^k_j$,
we have
\[
\bigl\|\xki-\txkm\bigr\|^2 = \biggl\|\sum_{j=0}^{t-1} \xi^k_j\biggr\|^2
\leq \biggl(\sum_{j=0}^{t-1}\bigl\|\xi^k_j\bigr\|\biggr)^2
\leq t \sum_{j=0}^{t-1} \bigl\|\xi^k_j\bigr\|^2,
\]
where the first inequality is due to the triangle inequality and the second
one is due to the Cauchy-Schwarz inequality.
Summing up for all $k=1,\ldots,K$ and $t=0,\ldots,m$, we get
\[
\sum_{k=1}^{K}\sum_{t=0}^{m-1}|\cbki| 
~ \leq ~ \epsilon^{-1} \sum_{k=1}^{K}\sum_{t=0}^{m-1}\biggl(t\sum_{j=0}^{t-1}\bigl\|\xi^k_j\bigr\|^2\biggr) 
~ \leq ~ \epsilon^{-1} \sum_{k=1}^{K}\sum_{t=0}^{m-1}\biggl(m\sum_{j=0}^{m-1}\bigl\|\xi^k_j\bigr\|^2\biggr) 
~ = ~ \epsilon^{-1}m^2 \sum_{k=1}^{K}\sum_{j=0}^{m-1}\|\xi^k_j\|^2 .
\]
Using H\"older's inequality, we have
\[
\sum_{k=1}^{K}\sum_{j=0}^{m-1}\|\xi^k_j\|^2
~\leq~ \biggl(\sum_{k=1}^{K}\sum_{j=0}^{m-1}1^3\biggr)^{1/3}
\biggl(\sum_{k=1}^{K}\sum_{j=0}^{m-1}\Bigl(\bigl\|\xi^k_j\bigr\|^2\Bigr)^{3/2}\biggr)^{2/3}
=~(Km)^{1/3}\biggl(\sum_{k=1}^{K}\sum_{j=0}^{m-1}\bigl\|\xi^k_j\bigr\|^{3}\biggr)^{2/3}.
\]
Therefore,
\[
\sum_{k=1}^{K}\sum_{t=0}^{m-1}|\cbki| 
~\leq~ \epsilon^{-1}m^2 (Km)^{1/3}\biggl(\sum_{k=1}^{K}\sum_{j=0}^{m-1}\bigl\|\xi^k_j\bigr\|^{3}\biggr)^{2/3}
=~ \epsilon^{-3/2}m^2 \biggl(\sum_{k=1}^{K}\sum_{j=0}^{m-1}\bigl\|\xi^k_j\bigr\|^{3}\biggr)^{2/3},
\]
where we used $Km=\epsilon^{-3/2}$ as specified 
in~\eqref{thm:Hess-complexity-m}.
Taking expectation on both sides gives 
\begin{eqnarray}
\label{thm:Hess-complexity-1}
\mathbb{E}\left[\sum_{k=1}^{K}\sum_{t=0}^{m-1}|\cbki|\right] 
& \leq & \epsilon^{-3/2}m^2\, \E\left[\biggl(\sum_{k=1}^{K}\sum_{t=0}^{m-1}\|\xiki\|^3\biggr)^{2/3}\right] \nonumber\\
& \leq &  \epsilon^{-3/2}m^2 \left(\E\left[\sum_{k=1}^{K}\sum_{t=0}^{m-1}\|\xiki\|^3\right]\right)^{2/3} \nonumber\\
& \leq & \epsilon^{-3/2}m^2 \left(\frac{F(\tx^0) - F^\star}{\sigma/4 - 3\rho - L}+\frac{5\rho+4L/3}{\sigma/4 - 3\rho - L}Km\epsilon^{3/2}\right)^{2/3} \nonumber\\
& = & \epsilon^{-3/2}m^2 \left(\frac{F(\tx^0) - F^\star}{\sigma/4 - 3\rho - L}+\frac{5\rho+4L/3}{\sigma/4 - 3\rho - L}\right)^{2/3},  \nonumber
\end{eqnarray}
where the second inequality is due to Jensen's inequality and the third 
inequality is due to~\eqref{eqn:sum-xikt}.
 
In addition to the sum of $|\cbki|$ over~$k$ and~$t$, 
we also need to account for the $N$ samples used to compute $\thk$
at the beginning of each state~$k=1,\ldots,K$.
Therefore, the total Hessian sample complexity is 
\[
\mathbb{E}\left[B_H\right] 
~=~ \mathbb{E}\left[\,\sum_{k=1}^{K}\sum_{i=0}^{m-1}|\cbki|\right] + KN 
~\leq~ C\epsilon^{-3/2}m^2 + \epsilon^{-3/2}N/m,
\]
where we used $Km=\epsilon^{-3/2}$ and the constant~$C$ is defined as
$$C = \left(\frac{F(\tx^0) - F^\star}{\sigma/4 - 3\rho - L}+\frac{5\rho+4L/3}{\sigma/4 - 3\rho - L}\right)^{2/3}.$$ 
By taking 
\[
m = (N/2C)^{1/3} = \mathcal{O}(N^{1/3}),
\]
we have minimized the upper bound on $\mathbb{E}\left[B_H\right]$ and achieved
\[
\mathbb{E}\left[B_H\right] 
~\leq~ N^{2/3}\epsilon^{-3/2}C^{1/3}\bigl(2^{-2/3} + 2^{1/3}\bigr) 
~=~ \mathcal{O}\bigl(N^{2/3}\epsilon^{-3/2}\bigr),
\]
which is the desired result.
\end{proof}  

Note that the above bound holds only when $\epsilon$ is small enough so that~\eqref{thm:Hess-complexity-m} makes sense. In order to cover the case when $\epsilon$ is large, we can write 
$$\mathbb{E}\left[B_H\right] 
~\leq~ \mathcal{O}\bigl(N+N^{2/3}\epsilon^{-3/2}\bigr).$$
 Through similar arguments, one can bound the total gradient sample complexity $B_G$ by 
 $$\E[B_G]\leq \cO(N + N^{2/3}\epsilon^{-5/2}).$$
 Note that the complexity $\E[B_G]$ can be improved by sampling without replacement. 

\subsection{Analysis of sampling without replacement}
In this section, we show that sampling without replacement will not change
the sample complexity of Algorithm~\ref{algo:SVRC}. 
Different Hessian sample complexity bounds for sampling with and without 
replacement are derived in \cite{SVRC-Lan} (see Table~\ref{Table:complexity}),
and both of them are worse than the $\mathcal{O}(N^{2/3}\epsilon^{-2/3})$
complexity obtained in this paper.
Again, we start from the variance estimation of the subsampled gradient and Hessian.
\begin{lemma}
\label{lemma:var-without}
Suppose $X_1,\ldots,X_N$ are matrices in $\R^{d\times d}$ satisfying
$\frac{1}{N}\sum_{i=1}^NX_i = 0$.
Let $Z_1,...,Z_n$, where $n\leq N$, be uniformly sampled from $X_1,...,X_N$ 
without replacement. Then
\begin{eqnarray}
\mathbb{E}\left[\bigg\|\frac{1}{n}\sum_{i=1}^nZ_i\bigg\|_F^2\right]
&=& \frac{1}{n }\left(1-\frac{n-1}{N-1}\right)\mathbb{E}\left[\|Z_1\|_F^2\right], 
\label{lm:var-without-2} \\
\mathbb{E}\left[\bigg\|\frac{1}{n}\sum_{i=1}^nZ_i\bigg\|_F^4\right] 
&=& \frac{1}{n^4}\left(r_1\mathbb{E}\left[\|Z_1\|_F^4\right] + r_2\mathbb{E}\left[\bigl\langle Z_1,Z_2\bigr\rangle^2\right] + r_3\mathbb{E}\left[\|Z_1\|_F^2 \|Z_2\|_F^2\right]\right),
\label{lm:var-without-4}
\end{eqnarray}
where 
$\langle Z_i, Z_j\rangle = \mathrm{trace}\bigl(Z_i^T Z_j\bigr)$
and
\begin{eqnarray*}
r_1 &=& n\left(1-4\cdot\frac{n-1}{N-1}+6\cdot\frac{(n-1)(n-2)}{(N-1)(N-2)} -3\cdot\frac{(n-1)(n-2)(n-3)}{(N-1)(N-2)(N-3)}\right),\\
r_2 &=& n(n-1)\left(2-4\cdot\frac{n-2}{N-2}+2\cdot\frac{(n-2)(n-3)}{(N-2)(N-3)}\right),\\
r_3 &=& n(n-1)\left(1-2\cdot\frac{n-2}{N-2}+1\cdot\frac{(n-2)(n-3)}{(N-2)(N-3)}\right).
\end{eqnarray*} 
\end{lemma}
The proof of this lemma is given in Appendix~\ref{appendix:var-without}.
As a result of this lemma, we have the following corollary.
%, which is a standard result of Lemma \ref{lemma:var-without}. 
The proof is similar to that of Lemma~\ref{lemma:var-with-Hess-g}, which
we omit for simplicity.

\begin{corollary}
\label{corollary:var-without-4-Hess}
Let $\tgki$ and $\thki$ be constructed by~\eqref{eqn:svr-g} and~\eqref{eqn:svr-H} respectively, with the mini-batches $\cski$ and $\cbki$ be sampled without replacement. Then 
\beaa
\mathbb{E}\left[\bigl\|\tgki-\nabla F(\xki)\bigr\|^2\big|\xki\right]  
~\leq~ \frac{L^2}{|\cski|} \left(1- \frac{|\cski|-1}{N-1}\right)\bigl\|\xki-\txkm\bigr\|^2,
\eeaa
and
\beaa
\mathbb{E}\left[\bigl\|\thki-\nabla^2 F(\xki)\bigr\|_F^4\big|\xki\right]  
~\leq~ \frac{11\rho^4}{|\cbki|^2} \left(3-6\frac{|\cbki|-2}{N-2} + 3\frac{\bigl(|\cbki|-2\bigr)\bigl(|\cbki|-3\bigr)}{(N-2)(N-3)}\right)\bigl\|\xki-\txkm\bigr\|^4.
\eeaa
\end{corollary}

%\begin{proof}
%	Define $X_j = \nabla^2f_j(\xki) - \nabla^2f_j(\txkm) - \nabla^2f(\xki) + \nabla^2f(\txkm)$ for $j = 1,...,N$. Let $Z_r, r\in\cbki$ be randomly sampled from $X_1,...,X_N$ without replacement. Then due to \eqref{lm:var-with-Hess-g-1}, $\E[\|Z_1\|^4|\xki] \leq 11\rho^4\|\xki-\txkm\|^4$. Consequently, 
%	\beaa
%	\E[\|Z_1\|^2\|Z_2\|^2|\xki] \leq \half\E[\|Z_1\|^4+\|Z_2\|^2|\xki] = \E[\|Z_1\|^4|\xki] \leq 11\rho^4\|\xki-\txkm\|^4,
%	\eeaa
%   \beaa	
%	\E[(\langle Z_1,Z_2\rangle)^2|\xki]\leq \E[\|Z_1\|^2\|Z_2\|^2|\xki] \leq 11\rho^4\|\xki-\txkm\|^4.
%	\eeaa
%	Substituting these bounds to \eqref{lm:var-without-4}, we get 
%	$$\mathbb{E}\left[\|\thki-\nabla^2f(\xki)\|^4|\xki\right]  \leq \frac{11\rho^4}{|\cbki|^2} \left(3-6\frac{|\cbki|-2}{N-2} + 3\frac{(|\mathcal{B}_k|-2)(|\cbki|-3)}{(N-2)(N-3)}\right)\|\xki-\txkm\|^4.$$
%	When both $|\cbki|,N\gg 1$, then $\frac{|\cbki|-2}{N-2}\approx \frac{|\cbki|}{N}$, $\frac{(|\cbki|-2)(|\cbki|-3)}{(N-2)(N-3)}\approx \frac{|\cbki|^2}{N^2}$. Hence we prove the lemma. 
%\end{proof}

When $N>|\mathcal{B}_k|\gg1$, we have
$$
\frac{|\cbki|-2}{N-2}\approx \frac{|\cbki|}{N}, \qquad
\frac{(|\cbki|-2)(|\cbki|-3)}{(N-2)(N-3)}\approx \frac{|\cbki|^2}{N^2}.
$$
Thus by Corollary~\ref{corollary:var-without-4-Hess},
the following inequality holds approximately 
\beaa
\mathbb{E}\left[\|\thki-\nabla^2f(\xki)\|^4|\xki\right]& \leq & 33\rho^4 \left(\frac{1}{|\cbki|}- \frac{1}{N } \right)^2\|\xki-\txkm\|^4.
\eeaa 
We can derive the desired mini-batch sizes for Hessian and gradient smapling by requiring
\beaa
\mathbb{E}\left[\bigl\|\thki - \nabla^2 F(\xki)\bigr\|_F^3\Big|\xki\right] 
%& \leq &  \left(\mathbb{E}\left[\|\thki - \nabla^2f(\xki)\|^4\right]\right)^{\frac{3}{4}}  
&\leq& (33)^{3/4}\rho^3\left(\frac{1}{|\cbki|}-\frac{1}{N}\right)^{3/2}\bigl\|\xki-\txkm\bigr\|^3 \\
&\leq& \rho^3\max\bigl\{\bigl\|\xikim\bigr\|^3,\,\epsilon^{3/2}\bigr\},
\eeaa
and
\beaa
\E\left[\bigl\|\tgki-\nabla F(\xki)\bigr\|_F^{3/2}\Big|\xki\right] 
&\leq& L^{3/2}\left(\frac{1}{|\cski|}-\frac{1}{N}\right)^{3/4}\bigl\|\xki-\txkm\bigr\|^{3/2}\\
&\leq& L^{3/2}\max\bigl\{\bigl\|\xikim\bigr\|^3,\,\epsilon^{3/2}\bigr\}.
\eeaa
Namely, 
%$$\left(\frac{1}{|\mathcal{B}_k|} - \frac{1}{N}\right)^\frac{3}{2} \leq c\frac{\|\xi_k\|^3}{\|x_k-\tilde{x}\|^3},$$
%that is 
%$$ \frac{1}{|\mathcal{B}_k|} \leq \frac{1}{N} +  c\frac{\|\xi_k\|^2}{\|x_k-\tilde{x}\|^2}.$$
%Hence we can choose 
$$
|\cbki| ~\geq~ \frac{1}{\displaystyle\frac{1}{N} + \frac{\max\{\|\xikim\|^2,\epsilon\}}{\sqrt{33}\,\bigl\|\xki-\txkm\bigr\|^2}} ,
\quad \quad
|\cski| ~\geq~ \frac{1}{\displaystyle\frac{1}{N} +  \frac{\max\{\|\xikim\|^4,\epsilon^2\}}{\bigl\|\xki-\txkm\bigr\|^2}}.
$$

For the Hessian sample complexity, if we want to have a sublinear sample size 
$N^\alpha$ with $\alpha<1$, then we should expect 
$|\cbki| \leq \mathcal{O}(N^\alpha)$ holds during the iterations. 
This requires
$$
\frac{\max\bigl\{\|\xikim\|^2,\,\epsilon\bigr\}}{\sqrt{33}\,\bigl\|\xki-\txkm\bigr\|^2} 
%~=~ N^{-\alpha} - \frac{1}{N} ~=~ \cO(N^{-\alpha}) ~\gg~ \frac{1}{N}.
~\geq~\Omega(N^{-\alpha}) ~\gg~ \frac{1}{N},
$$
which further implies
$$
|\cbki| 
%~=~ \cO(N^\alpha)
~\approx~ \frac{\sqrt{33}\,\bigl\|\xki-\txkm\bigr\|^2}{\max\bigl\{\|\xikim\|^2,\,\epsilon\bigr\}}.
$$
Notice that this Hessian batch size is the same order as the one used by
sampling with replacement. 
Therefore no improvement should be expected in terms of the dependence on $N$. 
For gradient sampling, similar arguments lead to an uper bound of 
$\min\big\{N,\frac{\|\xki - \txkm\|^2}{\epsilon^2}\big\}$. 
Based on the above analysis, we provide a corollary on the sample complexity bounds when sampling without replacement is adopted. 

\begin{corollary} 
Consider Algorithm \ref{algo:SVRC}, where we sample $\cski$ and $\cbki$ 
without replacement and set the length of each epoch and number of epochs to be
$m = \mathcal{O}(N^{1/3})$ and $K = \epsilon^{-3/2}/m$ respectively.
Then the totoal number of Hessian samples $B_H$ and total number of gradient
samples $B_G$ required to reach a second-order 
$\epsilon$-solution is 
\begin{eqnarray*}
\mathbb{E}\left[B_H\right] 
&\leq& \mathcal{O}\left(N+N^{2/3}\epsilon^{-3/2}\right), \\
\mathbb{E}\left[B_G\right] 
&\leq& \mathcal{O}\left(N + N^{2/3}\epsilon^{-3/2}\min\big\{N^{1/3},\epsilon^{-1}\big\}\right).
\end{eqnarray*}
\end{corollary} 
%However, this will not improve the sample complexity of the Hessian matrix in the order of $N$. 

\section{Analysis of non-adaptive SVRC Schemes}
\label{sec:non-adaptive}

In this section, we consider variants of the SVRC methods that use fixed 
gradient and Hessian sample sizes across all iterations. 
In the first variant, we use the full gradients but subsampled Hessians.
In this case, we show that the total Hessian sample complexity is still 
$\mathcal{O}(N^{2/3}\epsilon^{-3/2})$.
The second variant uses both approximate gradients and approximate Hessians,
and adds a correction term to the gradient approximation based on the second-order
information. 
This variant is proposed in \cite{SVRC-Zhou}, where an 
$\widetilde{\mathcal{O}}(N^{4/5}\epsilon^{-3/2})$ sample complexity is proved
for both the gradient and Hessian approximations.
We obtain the same order of sample complexity using the higher moment bounds 
developed in this paper (instead of using concentration inequalities),
which avoid the $poly(\log d)$ factor and excessively large constant 
in the results of \cite{SVRC-Zhou}.

\subsection{The case of using full gradient}

\begin{algorithm2e}[t]
\linespread{1.2}\selectfont
	%\caption{Stochastic Variance Reduced Cubic Regularization Method (Full Gradient)}
	\caption{Non-adaptive SVRC method using full gradients}
	\label{algo:SVRC-full-gradient}%\LinesNumberedHidden
	{\it Input:} An initial point $x^0$, $\sigma>0$, sample size $B$, and inner loop length $m$. An estimate of the Lipschitz constant $\rho$ and a parameter $\gamma = \Theta(\rho)$.  \\
	\For{k = 1,......,K }{
		Compute $\thk = \nabla^2f(\txkm).$\\
		Assign $x^k_0 = \txkm.$\\
		\For{$t = 0,...,m-1$}{
        Randomly sample an index set $\mathcal{B}_t^k$ with constant cardinality $|\mathcal{B}_t^k|=B$. \\
        Construct $\thki$ according to \eqref{eqn:svr-H}. \\
        %with $|\mathcal{B}_t^k|=??$.\\
            Solve $\xiki = \displaystyle\arg\min_{\xi} \Bigl\{\xi^T\nabla F(\xki) + \frac{1}{2}\xi^T\thki\xi + \frac{\sigma}{6}\|\xi\|^3 \Bigr\}. $\\
		    $\xkie = \xki + \xiki.$}		
		$\tx^k = x^k_m.$
	}
\emph{Output:} $x^{k^*}_{t^*+1}$ with \\
$\quad$ 
	Option 1: %$x^{k^*}_{t^*+1}$ where
   $ 
        k^*,t^*=\displaystyle\argmin_{1\leq k\leq K,~0\leq t\leq m-1}\left(\textstyle\bigl\|\xiki\bigr\|^3+\frac{\rho/\gamma}{2B^{3/2}}\bigl\|\xki-\txkm\bigr\|^3\right) 
   $.\\ 
$\quad$ 
	%Option 2: Randomly choose $k^*,t^*$ and output $x^{k^*}_{t^*+1}$.
	Option 2: %$x^{k^*}_{t^*+1}$ where 
    $k^*$ and $t^*$ are chosen randomly.
\end{algorithm2e} 

Algorithm~\ref{algo:SVRC-full-gradient} describes the SVRC method using
full gradient in each iteration, i.e., $\tgki = \nabla F(\xki)$. 
One remark regarding output option 1 is that the best choice of $\gamma$, depending on the parameters $B, m, \sigma$ and $\rho$, is given in Lemma~\ref{lemma:gamma-full-gradient}. However, one does not need to know the value exactly, and any choice of $\gamma = \Theta(\rho)$ will not affect our complexity result.

Similar to the derivation of \eqref{lm:descent-finite-2}, we have the following result,
\bea
\E\bigl[F(\xkie)\bigr] 
& \leq & \E\bigl[F(\xki)\bigr] + \half\E\left[\bigl\|\nabla^2 F(\xki)-\thki\bigr\|_F\bigl\|\xiki\bigr\|^2\right] - \left(\frac{\sigma}{4} - \frac{\rho}{6}\right)\E\left[\bigl\|\xiki\bigr\|^3\right] \nonumber\\
& \leq & \E\bigl[F(\xki)\bigr] - \left(\frac{\sigma}{4} - \frac{\rho}{2}\right)\E\left[\bigl\|\xiki\bigr\|^3\right] + \frac{1}{6\rho^2}\E\left[\bigl\|\nabla^2 F(\xki)-\thki\bigr\|_F^3\right]  \nonumber\\
&\leq& \E\bigl[F(\xki)\bigr] - \left(\frac{\sigma}{4} - \frac{\rho}{2}\right)\E\left[\bigl\|\xiki\bigr\|^3\right] + \frac{1}{6\rho^2} \cdot \frac{33^{3/4}\rho^3}{|\cbki|^{3/2}} \E\left[\bigl\|\xki-\txkm\bigr\|^3\right]\nonumber\\
&\leq & \E\bigl[F(\xki)\bigr] - \left(\frac{\sigma}{4} - \frac{\rho}{2}\right)\E\left[\bigl\|\xiki\bigr\|^3\right] + \frac{5\rho}{2|\cbki|^{3/2}} \E\left[\bigl\|\xki-\txkm\bigr\|^3\right] .
\label{descent:full-gradient}
\eea
Note that the expectation is not taken over $|\cbki|$ since it is a predetermined constant, which we denote as $B$ from now on. 
Let us define a Lyapunov function 
\be
\label{defn:Lyapunov-finite}
R^{k}_t: = \E\left[F(\xki) + c_t\bigl\|\xki-\txkm\bigr\|^3\right],
\ee
where the coefficients $c_t$ are constructed recursively by setting $c_m=0$ and 
\be
\label{defn:ci-finite}
c_{t} = c_{t+1}\left(1+2\theta_1^{-3} + \theta_2^{-6}\right)+\frac{3\rho}{B^{3/2}}, \qquad t = 0,....,m-1.
\ee
Here $\theta_1,\theta_2$ are some constant to be determined later. 
Next, we prove a monotone decreasing property of this
Lyapunov function over one epoch of Algorithm~\ref{algo:SVRC-full-gradient}. 

First, we note the following simple fact, which is proved in 
Appendix~\ref{appendix:fact-recursive}.
\begin{lemma}
	\label{lemma:fact-recursive}
	For any $a,b,\theta_1,\theta_2>0$, the following inequality holds:
	\bea
	\label{lm:fact-recursive}
	(a+b)^3 \leq \left(1+2\theta_1^{-3} + \theta_2^{-6}\right) a^3 + \left(1 + \theta_1^6 + 2\theta_2^3\right)b^3.
	\eea
\end{lemma} 
Consequently, by substituting $a = \|\xki-\txkm\|, b = \|\xkie-\xki\|$ into Lemma \ref{lemma:fact-recursive} and then taking the expectation on both sides, we get 
\bea
\label{Lyapunov:recursive}
\E\bigl[\|\xkie-\txkm\|^3\bigr] & \leq & \E\left[\bigl(\|\xkie-\xki\| + \|\xki-\txkm\|\bigr)^3\right]\\
& \leq & \left(1+2\theta_1^{-3} + \theta_2^{-6}\right) \E\bigl[\|\xki-\txkm\|^3\bigr] + \left(1 + \theta_1^6 + 2\theta_2^3\right)\E\bigl[\|\xkie-\xki\|^3\bigr].\nonumber
\eea
Substituting \eqref{Lyapunov:recursive} into 
\eqref{descent:full-gradient} yields
	\beaa
	\E\left[F(\xkie) + c_{t+1}\bigl\|\xkie-\txkm\bigr\|^3\right] 
    &\leq & \E\bigl[F(\xki)\bigr] - \left(\frac{\sigma}{4} - \frac{\rho}{2} - c_{t+1}\left(1 + \theta_1^6 + 2\theta_2^3\right)\right)\E\bigl[\|\xkie-\xki\|^3\bigr] \\
	& & +\left(c_{t+1}\left(1+2\theta_1^{-3} + \theta_2^{-6}\right)+\frac{5\rho}{2B^{3/2}}\right) \E\bigl[\|\xki-\txkm\|^3\bigr] .
	\eeaa
Now, using the definition of $R^{k}_i$ in~\eqref{defn:Lyapunov-finite} 
and expression of $c_i$~\eqref{defn:ci-finite}, we see that the above 
inequality leads to the following descent property of the Lyapunov function. 

\begin{lemma}
	\label{lemma:Lyapunov-full-gradient}
Let the sequence $\{\xki\}$ be generated by 
Algorithm~\ref{algo:SVRC-full-gradient}, then for all~$k$ 
and $0\leq t\leq m-1$,
	\bea
	\left(\frac{\sigma}{4} - \frac{\rho}{2} - c_{t+1}\left(1 + \theta_1^6 + 2\theta_2^3\right)\right)\E\bigl[\|\xiki\|^3\bigr] + \frac{\rho}{2B^{3/2}} \E\bigl[\|\xki-\txkm\|^3\bigr] &\leq & R^k_t-R^k_{t+1}.
	\eea
\end{lemma} 

Let us define a constant
\begin{equation}\label{eqn:gamma-def}
\gamma = \displaystyle\min_{1\leq i\leq m}\left\{\frac{\sigma}{4} - \frac{\rho}{2} - c_{i+1}(1 + \theta_1^6 + 2\theta_2^3)\right\}.
\end{equation} 
Then Lemma~\ref{lemma:Lyapunov-full-gradient} implies
\[
	\sum_{t=0}^{m-1}\E\left[\bigl\|\xiki\bigr\|^3+\frac{\rho/\gamma}{2B^{3/2}}\bigl\|\xki-\txkm\bigr\|^3\right] ~\leq~ \frac{R_0^k-R_m^k}{\gamma}.
\]
Next, we prove that when the parameters are properly choosen, then $\gamma = \cO(\rho)$.
\begin{lemma}
	\label{lemma:gamma-full-gradient}
    Suppose we set the batch size $B = \alpha N^{2/3}$, where $\alpha\geq 8$ is a constant.
    If we set $\sigma \geq 3\rho$, $m = (1/3)N^{1/3}$, $\theta_1 = N^{1/9}$ and $\theta_2 = N^{1/18}$, then 
	$$
    \gamma ~\geq~ \frac{1-6e\alpha^{-3/2}}{3}\rho ~=~ \cO(\rho).
    $$
\end{lemma}
\begin{proof}
	By \eqref{defn:ci-finite} and the values of $\theta_1$, $\theta_2$, $B$ 
    and $m$, we have 
	\beaa
	c_t & = & \left(1+2N^{-1/3} + N^{-1/3}\right)c_{t+1} + 3\rho\alpha^{-3/2}N^{-1}\\
    &=&  \left(1+3N^{-1/3}\right)c_{t+1} + 3\rho\alpha^{-3/2}N^{-1} .
	\eeaa
Adding $\frac{\rho}{\alpha^{3/2} N^{2/3}}$ to both sides of the above equality,
we obtain
    $$
    c_t + \frac{\rho}{\alpha^{3/2} N^{2/3}}  ~=~  \left(1+3N^{-1/3}\right)\left(c_{t+1}+\frac{\rho}{\alpha^{3/2} N^{2/3}}\right).
    $$
	Thus, 
	\beaa
\max_{0\leq t\leq m} c_t ~=~ c_0 \leq \left(c_m+\frac{\rho}{\alpha^{3/2} N^{2/3}}\right)\left(1+3N^{-1/3}\right)^m ~\leq~ \frac{e\rho}{\alpha^{3/2} N^{2/3}},
	\eeaa
    where we used $c_m=0$ and by choosing $m={(1/3)N^{1/3}}$ we have $3N^{-1/3}=1/m$ and $(1+1/m)^m\leq e$.
    In addition,
	\beaa
	\gamma ~\geq~ \frac{\sigma}{4}-\frac{\rho}{2} - \left(1+N^{2/3} + 2N^{1/6}\right)c_0 ~\geq~ \frac{1-8e\alpha^{-3/2}}{4}\rho.
	\eeaa
Therefore if $\alpha \geq 8 >(8e)^{2/3}$, then we have $\gamma = \cO(\rho)$.
\end{proof}
\begin{theorem}
	\label{theorem:cvg-finite-full-gradient}
Suppose in Algorithm~\ref{algo:SVRC-full-gradient}
we set $m = (1/3)N^{1/3}$, $|\cbki| =B= 8N^{2/3}$, $\sigma\geq 3\rho$,
and $\gamma$ is defined in~\eqref{eqn:gamma-def}.
Let $k^*,t^*$ be given by either of the two output options in the algorithm, 
then
\begin{equation}\label{eqn:fix-B-first-bound}
\E\left[\bigl\|\xistar\bigr\|^3+\frac{\rho/\gamma}{2B^{3/2}}\bigl\|\xstar-\tx^{k^*-1}\bigr\|^3\right]  ~\leq~ \frac{F(\tx^0) - F^\star}{Km\gamma}.
\end{equation}
Consequently, we have 
\begin{eqnarray*}
\E\left[\bigl\|\nabla F(\xkiestar)\bigr\|\right] 
&\leq& \cO\bigl((Km)^{-2/3}\bigr),\\
\E\left[\lambda_{\min}\bigl(\nabla^2 F(\xkiestar)\bigr)\right] 
&\geq& -\cO\bigl((Km)^{-1/3}\bigr),
\end{eqnarray*}
and the expected total Hessian samples required 
to reach a second-order $\epsilon$-solution
is $\cO(N^{2/3}\epsilon^{-3/2})$.
\end{theorem}
\begin{proof}
We start with the definition of $R_t^k$ in~\eqref{defn:Lyapunov-finite}.
%and \eqref{defn:ci-finite}, 
Since $x_0^k =\txkm$ and $c_m=0$, we have 
$$
R^{k}_0 = \E[F(x^{k}_0)] = \E[F(x^{k-1}_m)], \qquad R^k_m = \E[F(x^k_m)].
$$
Then Lemma \ref{lemma:Lyapunov-full-gradient} and Lemma \ref{lemma:gamma-full-gradient} immediately give 
	$$\sum_{k=1}^{K}\sum_{i=0}^{m-1}\E\left[\bigl\|\xiki\bigr\|^3+\frac{\rho/\gamma}{2B^{3/2}}\bigl\|\xki-\txkm\bigr\|^3\right]
    ~\leq~ \frac{F(\tx^0)-F^\star}{\gamma}.$$
Using similar arguments for the output option~1 and~2 in the proof of
Theorem~\ref{theorem:cvg-finite},
the first bound~\eqref{eqn:fix-B-first-bound} follows.

From~\eqref{eqn:fix-B-first-bound}, by Jensen's inequality, one simply gets
	$$
    \E\left[\bigl\|\xistar\bigr\|^2\right]\leq\left(\frac{F(\tx^0)-F^\star}{\gamma Km}\right)^{2/3},\qquad 
    \E\left[\bigl\|\xistar\bigr\|\right]\leq\left(\frac{F(\tx^0)-F^\star}{\gamma Km}\right)^{1/3},$$
and
\begin{eqnarray*}
\E\left[\bigl\|\xstar-\tx^{k^*-1}\bigr\|^2\right]  
&\leq& \biggl(\frac{2B^{3/2}\bigl(F(\tx^0)-F^\star\bigr)}{Km\rho}\biggr)^{2/3},\\
\E\left[\bigl\|\xstar-\tx^{k^*-1}\bigr\|\right]  
&\leq& \biggl(\frac{2B^{3/2}\bigl(F(\tx^0)-F^\star\bigr)}{Km\rho}\biggr)^{1/3}.
\end{eqnarray*}
Then following the proof of Lemma~\ref{lemma:relation-finite} in
Appendix~\ref{appendix:relation-finite}, 
more specifically~\eqref{eqn:grad-bound-xi-H}, we have
\beaa
	\E\left[\bigl\|\nabla F(\xkiestar)\bigr\|\right] 
    & \leq & \E\left[ \left(\rho+\frac{\sigma}{2}\right)\bigl\|\xistar\bigr\|^2 + \frac{1}{2\rho}\bigl\|\thh_{i^*}^{k^*}-\nabla^2 F(x^{k^*}_{t^*})\bigr\|_F^2\right] \nonumber\\
	& \leq & \E\left[\sqrt{2}\left(\rho+\frac{\sigma}{2}\right)\bigl\|\xistar\bigr\|^2 + \frac{\rho}{2B}\bigl\|x_{t^*}^{k^*} - \tx^{k^*-1}\bigr\|^2\right] \nonumber\\
	& \leq & \cO\bigl((Km)^{-2/3}\bigr).
	\eeaa 
Similarly, using~\eqref{eqn:grad-bound-xi-H},
one can bound the minimum eigenvalue of the Hessian as
	\beaa
	\E\left[\lambda_{\min}\bigl(\nabla^2 F(\xkiestar)\bigr)\right] 
    & \geq &  -\E\left[\left(\rho+\frac{\sigma}{2}\right)\left\|\xiki\right\| -\bigl\|\nabla^2 F(\xki)-\thki\bigr\|_F\right] \\
    & \geq & -\E\left[\left(\rho+\frac{\sigma}{2}\right)\bigl\|\xistar\bigr\| - \frac{\rho}{B^{1/2}}\bigl\|\xkiestar - \tx^{k^*-1}\bigr\|\right] \\
	&\geq& \cO\bigl((Km)^{-1/3}\bigr).
	\eeaa
Finally, by setting $B=8N^{2/3}$, $m=(1/3)N^{1/3}$ and $K=\epsilon^{-3/2}/m$, the total number of Hessian samples is
    \[
        mKB + KN = \mathcal{O}\bigl(N^{2/3}\epsilon^{-3/2}\bigr).
    \]
This concludes the proof.
\end{proof}

\subsection{The case of using subsampled gradient}
In this subsection, we analyze a non-adaptive SVRC scheme proposed 
in \cite{SVRC-Zhou}, which is shown in Algorithm~\ref{algo:SVRC-1st-gradient}. Similar to Algorithm~\ref{algo:SVRC-full-gradient}, one need not knwon the best value of $\gamma$ given in Lemma~\ref{lemma:gamma-no-full-gradient}, any choice of $\gamma = \Theta(\rho)$ will be acceptable.
In this scheme, the approximate gradients and Hessians are constructed as follows:
\begin{align}
\tgki &= \frac{1}{|\cski|}\sum_{j\in\cski}\left(\nabla f_j(\xki)-\nabla f_j(\txkm)+ \tg\right)+ \frac{1}{|\cski|}\sum_{j\in\cski} \left(\thk-\nabla^2f_j(\txkm)\right)\bigl(\xki-\txkm\bigr),
\label{eqn:svr-g-corrected} \\
\thki &= \frac{1}{|\cbki|}\sum_{i\in\cbki}\left(\nabla^2 f_i(\xki)-\nabla^2 f_i(\txkm)\right) + \thk. 
\label{eqn:svr-H-repeat} 
\end{align}
where  $\tg = \nabla F(\txkm)$ and $\thk = \nabla^2 F(\txkm)$.
Compared with~\eqref{eqn:svr-g} and~\eqref{eqn:svr-H},
the construction of $\thki$ is the same but $\tgki$ includes an additional second-order correction term.

\begin{algorithm2e}[t]
\linespread{1.2}\selectfont
	%\caption{Stochastic Variance Reduced Cubic Regularization Method (No Full Gradient)}
	\caption{Non-adaptive SVRC method using subsampled gradients and Hessians.}
	\label{algo:SVRC-1st-gradient}%\LinesNumberedHidden
	{\it Input:} An initial point $x^0$,  $\sigma>0$, sample sizes $S$ and $B$, and inner loop length $m$. An estimate of the Lipschitz constant $\rho$ and a parameter $\gamma = \Theta(\rho)$.   \\
	\For{k = 1,......,K }{
		Compute $\thk = \nabla^2f(\txkm).$\\
		Assign $x^k_0=\txkm$.\\
		\For{$t = 0,...,m-1$}{
            Randomly sample index sets $\cski$ and $\cbki$ with 
            $|\cski|=S$ and $|\cbki|=B$. \\
            Construct $\tgki$ according to~\eqref{eqn:svr-g-corrected} and
            $\thki$ according to \eqref{eqn:svr-H-repeat}. \\
        Solve $\xiki = \arg\min_{\xi} \left\{\xi^T\tgki + \frac{1}{2}\xi^T\thki\xi + \frac{\sigma}{6}\|\xi\|^3\right\}. $}		
		$\xkie = \xki + \xiki.$
	}
\emph{Output:} $x^{k^*}_{t^*+1}$ with \\
$\quad$ 
	Option 1: 
    $k^*,t^*=\displaystyle\argmin_{1\leq k\leq T,~0\leq t\leq m-1}\textstyle\left\{\bigl\|\xiki\bigr\|^3+\left(\frac{\rho/\gamma}{2B^{3/2}} +\frac{\rho/\gamma}{3\sqrt{2}S^{3/4}} \right)\bigl\|\xki-\txkm\bigr\|^3\right\}$. \\
$\quad$ 
	%Option 2: Randomly choose $k^*,t^*$ and output $x^{k^*}_{t^*+1}$.
	Option 2: 
    $k^*$ and $t^*$ are chosen randomly.
\end{algorithm2e} 

The following lemma is proved in Appendix~\ref{appendix:variance-1st-gd}.
\begin{lemma}
	\label{lemma:variance-1st-gd} 
	Let $\tgki$ be constructed by \eqref{eqn:svr-g-corrected} and $|\cski|=S$,
    then the following inequalities hold:
\begin{eqnarray}
    \mathbb{E}\left[\tgki|\xki\right]  &=& \nabla F(\xki),\label{lm:expect-g}\\
    \mathbb{E}\left[\bigl\|\tgki-\nabla F(\xki)\bigr\|_F^2 \Big| \xki\right] &\leq& \frac{\rho^2}{4S}\bigl\|\xki-\txkm\bigr\|^4.
        \label{lm:var-g}
\end{eqnarray}
\end{lemma}
By the above lemma and a discussion similar to that 
of~\eqref{lm:descent-finite-2} and~\eqref{descent:full-gradient}, we have 
\bea
\E\left[F(\xkie)\big|\xki\right] 
& \leq & F(\xki)  -\left(\frac{\sigma}{4} - \frac{5\rho}{6}\right)\E\left[\bigl\|\xiki\bigr\|^3\Big|\xki\right] + \frac{1}{6\rho^2}\E\left[\bigl\|\nabla^2 F(\xki)-\thki\bigr\|_F^3\Big|\xki\right] \nonumber\\
&& + \frac{2}{3\rho^{1/2}}\E\left[\bigl\|\nabla F(\xki) - \tgki\bigr\|^{3/2}\Big|\xki\right]
\label{lm:descent-finite-3} \\
& \leq &  F(\xki)  -\left(\frac{\sigma}{4} - \frac{5\rho}{6}\right)\E\left[\bigl\|\xiki\bigr\|^3\Big|\xki\right] + \biggl(\frac{5\rho}{2B^{3/2}} + \frac{\rho}{3\sqrt{2}S^{3/4}}\biggr)\bigl\|\xki-\txkm\bigr\|^3,\nonumber
\eea
where the last inequality is due to Lemmas~\ref{lemma:var-with-Hess-g} 
and~\ref{lemma:variance-1st-gd} and Jensen's inequality. 
In the rest of the analysis, we use the Lyapunov function $R_t^k$ defined 
in~\eqref{defn:Lyapunov-finite} but with a new set of parameters  
by setting $c_m=0$ and
\bea
\label{defn:ci-new}
 c_{t} = c_{t+1}\left(1+2\theta_1^{-3} + \theta_2^{-6}\right)+\frac{3\rho}{B^{3/2}} +\frac{\sqrt{2}\rho}{3S^{3/4}}, \qquad t = 0,....,m-1.
\eea
Again the constants $\theta_1$ and $\theta_2$ will be determined later.
We present the following results without proof due to the similarity to 
their counterparts in previous subsection. 
\begin{lemma}
	\label{lemma:Lyapunov-no-full-gradient}
	Suppose the sequence $\{\xki\}_{t = 1,...,m}^{k = 1,...,K}$ is generated by Algorithm \ref{algo:SVRC-1st-gradient}, then we have
	\beaa
	\left(\frac{\sigma}{4} - \frac{5\rho}{6} - c_{t+1}\bigl(1 + \theta_1^6 + 2\theta_2^3\bigr)\right)\E\bigl[\|\xiki\|^3\bigr] + \left(\frac{\rho}{2B^{3/2}} +\frac{\rho}{3\sqrt{2}S^{3/4}} \right)\E\bigl[\|\xki-\txkm\|^3\bigr] 
    &\leq & R^k_t-R^k_{t+1}.
	\eeaa
    Define $\gamma = \displaystyle\min_{1\leq t\leq m}\textstyle\left\{\frac{\sigma}{4} - \frac{5\rho}{6} - c_{t+1}\bigl(1 + \theta_1^6 + 2\theta_2^3\bigr)\right\}$. Then for $k = 1,...,K$, 
	\bea
	\sum_{t=0}^{m-1}\E\left[\bigl\|\xiki\bigr\|^3+\left(\frac{\rho/\gamma}{2B^{3/2}} +\frac{\rho/\gamma}{3\sqrt{2}S^{3/4}} \right)\bigl\|\xki-\txkm\bigr\|^3\right] ~\leq~ \frac{R_0^k-R_m^k}{\gamma}.
	\eea
\end{lemma} 
\begin{lemma}
	\label{lemma:gamma-no-full-gradient}
If we set $\sigma \geq 4\rho$, $m = (1/3)N^{1/5}$, $B = \alpha N^{2/5}$, $S = \alpha^2N^{4/5}$, $\theta_1 = N^{1/15}$ and $\theta_2 = N^{1/30}$, then we have $\rho>0$ as long as $\alpha\geq12$ and
	$$
    \gamma ~\geq~ \frac{\rho}{6}\left(1-4\bigl(3+\sqrt{2}/3\bigr)e\alpha^{-3/2}\right) = \cO(\rho).
    $$
\end{lemma}
\begin{theorem}
	\label{theorem:cvg-finite-no-full-gradient}
	For Algorithm \ref{algo:SVRC-1st-gradient}, set 
$$
\sigma\geq 4\rho, \quad
m = (1/3)N^{1/5}, \quad B = 12N^{2/5}, \quad S = B^2.
$$ 
Let $k^*,t^*$ be chosen by either of the two output options in the algorithm, 
then
	$$\E\left[\bigl\|\xistar\bigr\|^3+\left(\frac{\rho/\gamma}{2B^{3/2}} + \frac{\rho/\gamma}{3\sqrt{2}S^{3/4}}\right)\bigl\|\xstar-\tx^{k^*-1}\bigr\|^3\right]  ~\leq~ \frac{F(\tx^0) - F^\star}{Km\gamma}.
    $$
Consequently, we have 
\begin{eqnarray*}
\E\left[\bigl\|\nabla F(\xkiestar)\bigr\|\right] 
&\leq& \cO\bigl((Km)^{-2/3}\bigr),\\
\E\left[\lambda_{\min}\bigl(\nabla^2 F(\xkiestar)\bigr)\right] 
&\geq& -\cO\bigl((Km)^{-1/3}\bigr),
\end{eqnarray*}
and the total Hessian and gradient sample complexity 
for finding a second-order $\epsilon$-solution is
$\cO\bigl(N^{4/5}\epsilon^{-3/2}\bigr)$. 
\end{theorem}

%The Hessian and gradient sample complexity obtained here has the same 
%order in terms of $N$ and $\epsilon$ as the result in~\cite{SVRC-Zhou}, 
%but with a much smaller constant.

\section{Discussion}

We considered the problem of minimizing the average of a large
number of smooth and possibly nonconvex functions, 
$F(x)=(1/N)\sum_{i=1}^N f_i(x)$,
using subsampled Newton method with cubic regularization. 
We presented an adaptive variance reduction method that requires
$\cO(N+N^{2/3}\epsilon^{-3/2})$ Hessian samples 
for finding an approximate solution satisfying $\|\nabla F(x)\|\leq\epsilon$ 
and $\nabla^2 F(x) \succeq -\sqrt{\epsilon} I$.
This result holds for both sampling with and without replacement. 
Our analysis do not rely on high probability bounds from matrix concentration 
inequalities, instead, we use bounds on 3rd and 4th moments
of the average of random matrices. 

We have focused on the Hessian sample complexity by assuming that the 
solution to the cubic regularization (CR) subproblem~\eqref{eqn:CR-sub}
is available at each iteration.
Nesterov and Polyak \cite{Cubic:Nesterov} showed that the CR subproblem
is equivalent to a convex one-dimensional optimization problem, 
but this approach requires eigenvalue decomposion of the approximate Hessian 
$H^k$, which can be very costly for large-dimensional problems. 
Several recent works propose to solve the CR subproblem using iterative 
algorithms such as gradient descent \cite{Cubic:Carmon,chi-jin} or 
Lanczos method \cite{Cubic:Cartis-1,SCR-1}, 
and approximate trust-region solver \cite{HazanKoren2016} can also be used.
However, the overall complexity of the combined methods may still be high.

An efficient approximate solver for the CR subproblem
has been developed in \cite{AgarwalAllenZhu2017},
which leads to a total computational complexity of 
$\cO(N\epsilon^{-3/2}+N^{3/4}\epsilon^{-7/4})$
for minimizing the finite average problem~\eqref{prob:main-finite}.
This complexity is measured in terms of the total number 
of Hessian-vector products, i.e., multiplications by the component 
Hessians $\nabla^2 f_i(x^k)$.
Similar results have also been obtained by \cite{ReddiZaheer2018aistat}.
For many mahcine learning problems, including generalized linear models
and training neural networks, such Hessian-vector products can be computed
in $\cO(d)$ time \cite{Pearlmutter1994,chi-jin}.
We notice that the CR subproblem solved in~\cite{AgarwalAllenZhu2017} 
includes all~$N$ component Hessians. Thus 
it is possible to further reduce the overall computational complexity
by combining the efficient CR solver developed in \cite{AgarwalAllenZhu2017}
and the lower Hessian sample complexity obtained in this paper.

\section*{Acknowledgments}
We thank Zeyuan Allen-Zhu for helpful discussions on the complexities of 
several first and second-order methods for nonconvex optimization 
and the approximate cubic regularization solver in \cite{AgarwalAllenZhu2017}.

\appendix
\section{Proof of Lemma \ref{lemma:var-with-pw4}}
\label{appendix:iid4thmoment}
\begin{proof}
    Let $\langle Z_1, Z_2\rangle=\mathrm{trace}(Z_1^T Z_2)$ be the inner product
    of two matrices. Then $\|Z_1\|_F^2 = \langle Z_1, Z_1\rangle$.
Using the assumption that $\mathbb{E}[Z_i]=0$ for all~$i$ and $Z_1,\ldots,Z_n$
are independent and identically distributed, we have
	\begin{equation}
	\label{lmpf:var-with-pw4-1}
	\mathbb{E}\left[\bigg\|\frac{1}{n}\sum_{i=1}^{n}Z_i\bigg\|_F^4\right]  
    = \frac{1}{n^4} \mathbb{E} \left[\biggl\langle \sum_{i=1}^n Z_i,\, \sum_{i=1}^n Z_i\biggr\rangle^2\right]
    = \frac{1}{n^4}(T_1+T_2+\cdots+T_7),
	\end{equation}
	where 
\begin{eqnarray*}
    T_1 &=& n\mathbb{E}\left[\|Z_1\|_F^4\right], \\
    T_2 &=& 4n(n-1)\mathbb{E}\left[\|Z_1\|_F^2\langle Z_1, Z_2\rangle\right] ~=~ 0, \\
    T_3 &=& 2n(n-1)\mathbb{E}\left[\langle Z_1, Z_2\rangle^2\right] ~\leq~ 2n(n-1)\mathbb{E}\left[\|Z_1\|_F^4\right], \\
    T_4 &=& n(n-1)\mathbb{E}\left[\|Z_1\|_F^2\|Z_2\|_F^2\right] ~\leq~ n(n-1)\mathbb{E}\left[\|Z_1\|_F^4\right], \\
    T_5 &=& 4n(n-1)(n-2)\mathbb{E}\left[\langle Z_1,Z_2\rangle\langle Z_2,Z_3\rangle\right] ~=~ 0, \\
    T_6 &=& 2n(n-1)(n-2) \mathbb{E}\left[\langle Z_1,Z_2\rangle\|Z_3\|_F^2\right]  ~=~ 0, \\
    T_7 &=& n(n-1)(n-2)(n-3)\mathbb{E}\left[\langle Z_1,Z_2\rangle\langle Z_3,Z_4\rangle\right] ~=~ 0. 
\end{eqnarray*}
In total, we have 
	$$\mathbb{E}\left[\bigg\|\frac{1}{n}\sum_{i=1}^{n}Z_i\bigg\|_F^4\right]  = \frac{3n^2-2n}{n^4}\mathbb{E}\left[\|Z_1\|_F^4\right] \leq \frac{3}{n^2}\mathbb{E}\left[\|Z_1\|_F^4\right].$$
which is the desired result.
\end{proof}

\section{ Proof of Lemma \ref{lemma:var-with-Hess-g}}
\label{appendix:Hess-grad-var}
\begin{proof}
First, we prove the bounds for the variance reduced Hessian estimate. 
It is straightforward to show that 
    $\mathbb{E}\left[\thki|\xki\right]  = \nabla^2 F(\xki)$.
    For the rest two inequalities, let us first define 
    $$Z_j = \nabla^2f_{j}(\xki) - \nabla^2f_{j}(\txkm) + \nabla^2 F(\txkm) - \nabla^2 F(\xki),$$ 
    where $j$ is a uniform sample from $\{1,...,N\}$.  
    Note that 
    $$\E\left[\nabla^2 f_{j}(\xki) - \nabla^2 f_{j}(\txkm) \,\Big|\, \xki\right] =  \nabla^2 F(\xki) - \nabla^2 F(\txkm).$$
Therefore, $\mathbb{E}\left[Z_j\big|\xki\right] = 0$.  
    For the ease of notation, define 
    $z_j = \nabla^2 f_{j}(\xki) - \nabla^2 f_{j}(\txkm)$ such that
$Z_j=z_j-\E\bigl[z_j|\xki\bigr]$. 
    According to the Lipschitz continuity conditions, 
    $\|z_j\|_F\leq\rho\|\xki-\txkm\|$. 
For the second moment, we have
    \begin{align*}
    \mathbb{E}\left[\|Z_j\|_F^2\Big|\xki\right] 
    & =  \E\left[\bigl\|z_j-\E[z_j|\xki]\bigr\|_F^2\Big|\xki\right] \\
    & = \E\left[\|z_j\|_F^2\big|\xki\right] - \bigl\|\E[z_j|\xki]\bigr\|_F^2 \\
    & \leq \E\left[\|z_j\|_F^2\big|\xki\right] \\ 
    & \leq  \rho^2\bigl\|\xki-\txkm\bigr\|^2.
    \end{align*}
Therefore,
    \[
    \mathbb{E}\left[\bigl\|\thki-\nabla^2 F(\xki)\bigr\|_F^2\big|\xki\right] 
    = \mathbb{E} \Biggl[\biggl\|\frac{1}{|\cbki|}\sum_{j\in\cbki}Z_j\biggr\|_F^2\bigg|\xki \Biggr] 
    =  \frac{1}{|\cbki|}\mathbb{E}\left[\|Z_1\|_F^2\big|\xki\right]  
    \leq  \frac{\rho^2}{|\cbki|} \bigl\|\xki-\txkm\bigr\|^2.
    \]
For the fourth moment, we have
    \begin{align}
    \mathbb{E}\left[\|Z_j\|_F^4\big|\xki\right] 
    &= \E\left[\bigl\|z_j-\E[z_j|\xki]\bigr\|_F^4\big|\xki\right] \\
    %&= \E\Bigl[\Bigl\langle z_j-\E[z_j|\xki],\,z_j-\E[z_j|\xki]\Bigr\rangle^2\Big|\xki\Bigr] \nonumber\\
    %&= \E\Bigl[\Bigl( \|z_j\|_F^2 -2\bigl\langle z_j,\,\E[z_j|\xki] \bigr\rangle + \bigl\|\E[z_j|\xki]\bigr\|_F^2\Bigr)^2 \Big|\xki\Bigr] \nonumber\\
   &= \E\left[\|z_j\|_F^4\big|\xki\right] 
     + 4\E\left[\bigl\langle z_j, \E\bigl[z_j|\xki\bigr]\bigr\rangle^2\big|\xki\right] 
     + 2\E\left[\|z_j\|_F^2\big|\xki\right]\bigl\|\E[z_j|\xki]\bigr\|_F^2\nonumber\\
	& \qquad -4\Bigl\langle \E\bigl[\|z_j\|_F^2 z_j|\xki\bigr],\, \E[z_j|\xki]\Bigr\rangle - 3\bigl\|\E[z_j|\xki]\bigr\|_F^4\nonumber\\
	&\leq 11\, \E\left[\|z_j\|_F^4\big|\xki\right]\nonumber\\
        &\leq 11  \rho^4\bigl\|\xki-\txkm\bigr\|^4, \nonumber
%\label{lm:var-with-Hess-g-1}
\end{align}
where we used 
$\langle z_j,\E[z_j|\xki]\rangle\leq\|z_j\|_F \bigl\|\E[z_j|\xki]\bigr\|_F$
and $\bigl\|\E[z_j|\xki]\bigr\|_F^2 \leq\E\bigl[\|z_j\|_F^2\big|\xki\bigr]$.
Then by Lemma~\ref{lemma:var-with-pw4}, we obtain
\[
\mathbb{E}\left[\bigl\|\thki-\nabla^2 F(\xki)\bigr\|_F^4\Big|\xki\right] 
 = \mathbb{E}\Biggl[\biggl\|\frac{1}{|\cbki|}\sum_{j\in\cbki}Z_j\biggr\|_F^4 \bigg|\xki\Biggr] 
\leq \frac{3}{|\cbki|^2}\mathbb{E}\left[\|Z_1\|_F^4\big|\xki\right]  
\leq  \frac{33\rho^4}{|\cbki|^2} \bigl\|\xki-\txkm\bigr\|^4.
\]
Similarly, one can get the variance bounds for the gradient estimate. 
This part of the proof is omitted for simplicity.
\end{proof}

\section{Proof of Corollary \ref{corollary:var-with-Hess-g-xi}}
\label{appendix:var-Hess-grad-xi}
\begin{proof}
By Lemma \ref{lemma:var-with-Hess-g} and the mini-batch size rule in 
Algorithm~\ref{algo:SVRC},
\[
\mathbb{E}\left[\bigl\|\thki-\nabla^2 F(\xki)\bigr\|_F^2\big|\xki\right] 
~\leq~ \frac{\rho^2}{|\cbki|} \bigl\|\xki-\txkm\bigr\|^2 
~\leq~ \rho^2\epsilon_H 
%= \rho^2\max\{\rho^{-4/3}\|\xikim\|^2,\epsilon^k_h\} 
~\leq~ \rho^2\left(\|\xikim\|^2 + \epsilon\right).
\]
Due to the concavity of the square root function $\sqrt{\cdot}$, 
we can apply Jensen's inequality to obtain
\[
\mathbb{E}\left[\bigl\|\thki-\nabla^2 F(\xki)\bigr\|_F\big|\xki\right] 
~\leq~ \sqrt{\mathbb{E}\left[\|\thki-\nabla^2 F(\xki)\|_F^2\big|\xki\right]}
~\leq~ \rho\sqrt{\epsilon_H} 
~\leq~ \rho\left(\|\xikim\| + \epsilon^{1/2}\right).
\]
Similarly, we have 
\[
\mathbb{E}\left[\bigl\|\thki-\nabla^2 F(\xki)\bigr\|_F^4\big|\xki\right] 
~\leq~ \frac{33\rho^4}{|\cbki|^2} \bigl\|\xki-\txkm\bigr\|^4 
~\leq~ 33\rho^4 \epsilon_H^2. 
%= \rho^2\max\{\rho^{-4/3}\|\xikim\|^2,\epsilon^k_h\} 
\]
Again, with the concavity of the function $(\cdot)^{3/4}$, applying Jensen's inequality yields
\begin{align*}
\mathbb{E}\left[\bigl\|\thki-\nabla^2 F(\xki)\bigr\|^3_F\big|\xki\right] 
&\leq \left(\mathbb{E}\left[\bigl\|\thki-\nabla^2 F(\xki)\bigr\|_F^4\big|\xki\right]\right)^{3/4} 
~\leq~ 33^{3/4}\rho^3\epsilon_H^{3/2} \\
&\leq~ 33^{3/4}\rho^3\left(\|\xikim\|^3+\epsilon^{3/2}\right).
\end{align*}

Following a similar line of arguments, one can get the bounds for the 
gradient variances. We omit the details for simplicity.
\end{proof}

\section{Proof of Lemma \ref{lemma:relation-finite}}
\label{appendix:relation-finite}
\begin{proof}
Using the triangule inequality, we have
\begin{eqnarray*}
\bigl\|\nabla F(\xkie)\bigr\|
&\leq& \bigl\| \nabla F(\xkie) - \nabla F(\xki) -\nabla^2 F(\xki)\xiki \bigr\|
   +\bigl\|\nabla^2 F(\xki)\xiki - \thki\xiki\bigr\| \\
&& +\bigl\| \nabla F(\xki)-\tgki\bigr\| + \bigl\|\tgki + \thki\xiki\bigr\|.
\end{eqnarray*}
By the Lipschitz continuity of $\nabla^2 F$ and the fact $\xkie=\xki+\xiki$,
\[
\bigl\|\nabla F(\xkie) - \nabla F(\xki) - \nabla^2 F(\xki)\xiki\bigr\|
\leq \frac{\rho}{2}\bigl\|\xiki\bigr\|^2.
\]
In addition, the optimality condition~\eqref{eqn:subopt} implies
\[
\bigl\|\tgki + \thki\xiki\bigr\| = \frac{\sigma}{2}\bigl\|\xiki\bigr\|^2.
\]
Consequently, 
\begin{eqnarray}
\bigl\|\nabla F(\xkie)\bigr\| 
& \leq & \frac{\sigma+\rho}{2}\bigl\|\xiki\bigr\|^2 + \bigl\|\thki-\nabla^2 F(\xki)\bigr\|_F\bigl\|\xiki\bigr\| + \bigl\|\nabla F(\xki)-\tgki\bigr\|\nonumber\\
& \leq & \left(\rho+\frac{\sigma}{2}\right)\bigl\|\xiki\bigr\|^2 + \frac{1}{2\rho}\bigl\|\thki-\nabla^2 F(\xki)\bigr\|_F^2 + \bigl\|\nabla F(\xki)-\tgki\bigr\|.
\label{eqn:grad-bound-xi-H}
\end{eqnarray}
Taking expectation on both sides of the above inequality and applying 
Corollary \ref{corollary:var-with-Hess-g-xi} yield
\[
\E\left[\bigl\|\nabla F(\xkie)\bigr\|\right]
~\leq~ \left(\rho+\frac{\sigma}{2}\right)\E\left[\bigl\|\xiki\bigr\|^2\right] 
+ \left(\frac{\rho}{2}+L\right)\left(\E\left[\bigl\|\xikim\bigr\|^2\right] 
+\epsilon\right),
\]
which is the first desired result. 
	
For the secnod inequality, the optimality condition~\eqref{eqn:subopt-psd} 
implies	
$$\lambda_{\min}(\thki) \geq -\frac{\sigma}{2}\bigl\|\xiki\bigr\|.$$
Therefore, according to the Lipschitz continuity of $\nabla^2 F$,
\begin{eqnarray}
\lambda_{\min}\bigl(\nabla^2 F(\xkie)\bigr) 
& \geq & \lambda_{\min}\bigl(\nabla^2 F(\xki)\bigr) - \rho\bigl\|\xiki\bigr\| 
\nonumber \\
& \geq & \lambda_{\min}\bigl(\thki\bigr) - \rho\bigl\|\xiki\bigr\| -\bigl\|\nabla^2 F(\xki)-\thki\bigr\|_F \nonumber \\ 
& \geq & - \left(\rho+\frac{\sigma}{2}\right)\bigl\|\xiki\bigr\| - \bigl\|\nabla^2 F(\xki)-\thki\bigr\|_F .
\label{eqn:eigen-bound-xi-H}
\end{eqnarray}
Taking expectation on both sides and applying the Corollary \ref{corollary:var-with-Hess-g-xi} results in 
\begin{eqnarray*}
 \mathbb{E}\left[\lambda_{\min}\bigl(\nabla^2 F(\xkie)\bigr)\right]  
   & \geq & - \left(\rho+\frac{\sigma}{2}\right)\mathbb{E}\left[\bigl\|\xiki\bigr\|\right] -\mathbb{E}\left[\bigl\|\nabla^2 F(\xki)-\thki\bigr\|_F\right] \\
   & \geq & -\left(\rho+\frac{\sigma}{2}\right)\mathbb{E}\left[\bigl\|\xiki\bigr\|\right] -\rho\left(\mathbb{E}\left[\|\xikim\|\right]+ \epsilon^{1/2}\right),
\end{eqnarray*}
which is the second desired inequality.
\end{proof}

\section{Proof of Lemma \ref{lemma:var-without}}
\label{appendix:var-without}
\begin{proof}
Equation \eqref{lm:var-without-2} is a standard result for variance analysis of sampling without replacement scheme. Hence we omit the proof of this equation. 
To prove \eqref{lm:var-without-4}, we start with the expansions 
in~\eqref{lmpf:var-with-pw4-1} and calculate the terms $T_1,...,T_7$ 
for sampling without replacement.
First, the following three terms do not change:
\beaa
T_1 & = & n\mathbb{E}\left[\|Z_1\|_F^4\right], \\
T_3 & = & 2n(n-1)\mathbb{E}\left[\bigl(\langle Z_1,Z_2\rangle\bigr)^2\right],\\
T_4 & = & n(n-1)\mathbb{E}\left[\|Z_1\|_F^2\|Z_2\|_F^2\right]. 
\eeaa
For $T_2$, we have
\beaa
T_2 & = & 4n(n-1)\mathbb{E}\left[\|Z_1\|_F^2\langle Z_1,Z_2\rangle\right] \\
%& = & 4n(n-1)\mathbb{E}\Bigl[\E\bigl[\|Z_1\|_F^2\langle Z_1,Z_2\rangle\bigl]\big|Z_1\Bigr] \\
& = & 4n(n-1)\mathbb{E}\Bigl[\|Z_1\|_F^2\bigl\langle Z_1,\E\left[Z_2|Z_1\right]\bigr\rangle\Bigr] \\
& = & 4n(n-1)\sum_{i=1}^N\frac{1}{N}\|X_i\|_F^2\bigg\langle X_i,\sum_{j\neq i}\frac{1}{N-1}X_j\bigg\rangle \\
& = & \frac{4n(n-1)}{(N)(N-1)} \sum_{i=1}^N\|X_i\|_F^2 \bigg\langle X_i,\underbrace{\sum_{j=1}^NX_j}_{0} - X_i\bigg\rangle\\[-3ex]
& = & -\frac{4n(n-1)}{ N-1 } \sum_{i=1}\frac{1}{N}\|X_i\|_F^4\\
& = & -\frac{4n(n-1)}{ N-1 } \mathbb{E}\left[\|Z_1\|_F^4\right]. 
\eeaa
For $T_5$, we have
\beaa
T_5 & = & 4n(n-1)(n-2)\mathbb{E}\bigl[\langle Z_1,Z_2\rangle\cdot\langle Z_2,Z_3\rangle\bigr]\\
& = & 4n(n-1)(n-2)\mathbb{E}\Bigl[\bigl\langle \E\bigl[Z_1|Z_2\bigr],Z_2\bigr\rangle\cdot\bigl\langle Z_2,\E\bigl[Z_3|Z_2,Z_1\bigr]\bigr\rangle\Bigr]\\
& = & \frac{4n(n-1)(n-2)}{N(N-1)(N-2)}\sum_{i=1}^N\bigg\langle X_i,\sum_{j\neq i}X_j\biggl(\Bigl\langle X_i,\sum_{k\neq i,j }X_k\Bigr\rangle\biggr)\bigg\rangle \\ 
& = & \frac{4n(n-1)(n-2)}{N(N-1)(N-2)}\sum_{i=1}^N\bigg\langle X_i,\sum_{j\neq i}X_j\biggl(\Bigl\langle X_i,-X_i-X_j\Bigr\rangle\biggr)\bigg\rangle \\ 
& = & \frac{4n(n-1)(n-2)}{N(N-1)(N-2)}\sum_{i=1}^N\bigg\langle X_i,\sum_{j\neq i}X_j\Bigl(-\|X_i\|_F^2-\langle X_i,X_j\rangle\Bigr)\bigg\rangle\\
& = & -\frac{4n(n-1)(n-2)}{N-2}\sum_{i=1}^N\sum_{j\neq i}\frac{1}{N(N-1)}\left(\|X_i\|_F^2\langle X_i,X_j\rangle + \bigl(\langle X_i,X_j\rangle\bigr)^2\right)\\
& = & -\frac{4n(n-1)(n-2)}{N-2}\Bigl(\mathbb{E}\left[\|Z_1\|_F^2\langle Z_1,Z_2\rangle\right] + \mathbb{E}\left[\bigl(\langle Z_1,Z_2\rangle\bigr)^2\right]\Bigr)\\
&  = & 4n\frac{(n-1)(n-2)}{(N-1)(N-2)}\mathbb{E}\left[\|Z_1\|_F^4\right]
-\frac{4n(n-1)(n-2)}{N-2} \mathbb{E}\left[\bigl(\langle Z_1,Z_2\rangle\bigr)^2\right],
\eeaa
where in the last equality we used result for $T_2$.
In similar ways, one can find the expressions for $T_6$ and $T_7$:
\beaa
T_6 & = & 2n(n-1)(n-2)\mathbb{E}\left[\|Z_3\|_F^2\langle Z_1,Z_2\rangle\right]\\
& = & 2n\frac{(n-1)(n-2)}{(N-1)(N-2)}\mathbb{E}\left[\|Z_1\|_F^4\right]  - 2n(n-1)\frac{n-2}{N-2}\mathbb{E}\left[\|Z_1\|_F^2\|Z_2\|_F^2\right]
\eeaa
and
\beaa
T_7& = & 2n(n-1)(n-2)(n-3) \mathbb{E}\bigl[\langle Z_1,Z_2\rangle\cdot\langle Z_3,Z_4\rangle\bigr]\\
& = & -3n\frac{(n-1)(n-1)(n-3)}{(N-1)(N-2)(N-3)} \mathbb{E}\left[\|Z_1\|_F^4\right]
+2n(n-1)\frac{(n-2)(n-3)}{(N-2)(N-3)}\mathbb{E}\bigl[\bigl(\langle Z_1,Z_2\rangle\bigr)^2\bigr]\\
& & + n(n-1)\frac{(n-2)(n-3)}{(N-2)(N-3)}\mathbb{E}\bigl[\|Z_1\|_F^2\|Z_2\|_F^2\bigr].
\eeaa
Summing these terms up gives the desired equation \eqref{lm:var-without-4}. 
\end{proof}

\section{Proof of Lemma \ref{lemma:fact-recursive}}
\label{appendix:fact-recursive}
\begin{proof} 
We expand $(a+b)^3$ and then use Young's inequality,
\beaa
	(a+b)^3 & = & a^3 + b^3 + 3a^2b + 3b^2a \\
	&=& a^3 + b^3 + 3(a/\theta_1)^2(b\theta_1^2) + 3(a/\theta_2^2)(b\theta_2)^2 \\
	& \leq & a^3 + b^3 + 3 \left( \frac{(a^2/\theta_1^2)^{\ttwo}}{3/2} + \frac{(b\theta_1^2)^3}{3}\right) + 3 \left( \frac{(a/\theta_2^2)^{3}}{3} + \frac{(b^2\theta_2^2)^{\ttwo}}{3/2}\right)\\
	& = & (1+2\theta_1^{-3} + \theta_2^{-6}) a^3 + (1 + \theta_1^6 + 2\theta_2^3)b^3 .
\eeaa
This completes the proof.
\end{proof}

\section{Proof of Lemma \ref{lemma:variance-1st-gd}}
\label{appendix:variance-1st-gd}
\begin{proof}
Define 
$$
Z_j = \nabla f_{j}(\xki) - \nabla f_j(\txkm) - \nabla^2f_j(\txkm)(\xki-\txkm) - \left(\nabla F(\xki) - \nabla F(\txkm) - \nabla^2 F(\txkm)(\xki-\txkm)\right)
$$ 
and 
$$
z_j = \nabla f_{j}(\xki) - \nabla f_j(\txkm) - \nabla^2f_j(\txkm)(\xki-\txkm),
$$
so that we have $Z_j = z_j-\E[z_j|\xki]$.
According to~\eqref{eqn:svr-g-corrected}, we have
$\tgki -\nabla F(\xki)= \frac{1}{|\cski|}\sum_{j\in\cski}Z_j$.
Therefore, $\E[\tgki|\xki] = \nabla F(\xki)$ and
\beaa
\E\left[\bigl\|\tgki -\nabla F(\xki)\bigr\|^2\big|\xki\right] 
& = &  \frac{1}{|\cski|}\E\left[\|Z_1\|^2\big|\xki\right] \\
& = &  \frac{1}{|\cski|}\left(\E\left[\|z_1\|^2\big|\xki\right] - \bigl\|\E[z_1|\xki]\bigr\|^2\right) \\
& \leq & \frac{1}{|\cski|}\E\left[\|z_1\|^2|\xki\right]\\
& \leq & \frac{\rho^2}{4|\cski|}\bigl\|\xki-\txkm\bigr\|^4
\eeaa
where the last inequality is due to the Lipschitz continuity of $\nabla^2f$. 
\end{proof}

\bibliographystyle{abbrv}
\bibliography{SVRC}

\end{document}